\documentclass[a4paper,10pt]{article}
\usepackage[utf8x]{inputenc}
\usepackage{tracefnt,amsmath,tabu,array}
\usepackage{amssymb,graphicx,setspace,amsfonts,amsbsy}
\usepackage{pifont,latexsym,ifthen,amsthm,rotating,calc,textcase,booktabs,cancel,slashed}

\newtheorem{theorem}{Theorem}[section]
\newtheorem{lemma}[theorem]{Lemma}
\newtheorem{corollary}[theorem]{Corollary}
\newtheorem{proposition}[theorem]{Proposition}

\newtheorem{remark}[theorem]{Remark}
\newcommand{\filledbox}{\leavevmode
  \hbox to.77778em{%
  \hfil\vbox to.675em{\hrule width.6em height.6em}\hfil}}

\newcommand{\Rm}{{\mathbb R}}

\begin{document}
\tabulinesep=1.0mm
\title{Nonexistence of multi-bubble radial solutions to the 3D energy critical wave equation}

\author{Ruipeng Shen\\
Centre for Applied Mathematics\\
Tianjin University\\
Tianjin, China
}

\maketitle

\begin{abstract}
 In this work we consider the focusing, energy-critical wave equation in 3D radial case. It has been verified that any global or type II blow-up solution decomposes into a superposition of several decoupled grounds states, a free wave and a small error, as time tends to infinity or the blow-up time. This is usually called soliton resolution. However, all known examples of soliton resolution in the 3D radial case come with no more than one soliton. In this work we prove the nonexistence of any global or type II blow up solution with two or more solitons, thus give a complete classification of asymptotic behaviours of radial solutions. 
\end{abstract}

\section{Introduction} 

\subsection{Background}

In this work we consider radial solutions to the following focusing, energy-critical wave equation in 3-dimensional space 
\[
 \left\{\begin{array}{ll} \partial_t^2 u - \Delta u = + |u|^4 u, & (x,t) \in \Rm^3 \times \Rm;  \\ (u,u_t)|_{t=0} = (u_0,u_1)\in \dot{H}^1\times L^2. & \end{array} \right. \qquad \hbox{(CP1)}
\]
The solutions satisfy the energy conservation law:
\[
 E = \int_{\Rm^3} \left(\frac{1}{2}|\nabla u(x,t)|^2 + \frac{1}{2}|u_t(x,t)|^2 - \frac{1}{6}|u(x,t)|^6\right) {\rm d} x = \hbox{Const}. 
\]
The equation is invariant under the dilation transformation. More precisely, if $u$ is a solution to (CP1), then 
\[ 
 u_\lambda = \frac{1}{\lambda^{1/2}} u\left(\frac{x}{\lambda}, \frac{t}{\lambda}\right), \qquad \lambda\in \Rm^+
\]
is also a solution to (CP1) with exactly the same energy. 

The local well-posedness of (CP1) follows from a combination of the standard fixed-point argument and suitable Strichartz estimates. This argument does not depend on the sign of the nonlinearity, thus applies to the defocusing wave equation $\partial_t^2 u - \Delta u = - |u|^4 u$ as well. More details can be found in  Kapitanski \cite{loc1} and  Lindblad-Sogge \cite{ls}, for examples. The global behaviour in the focusing case, however, is much more complicated than the defocusing case. In short, all finite-energy solutions to the defocusing equation are globally defined for all $t$ and scatter in both time directions. Readers may refer to \cite{mg1, enscatter1, enscatter2, ss1, ss2, struwe}, for instance.  Although solutions to (CP1) with small initial data still scatter, large solutions may blow up in finite time or exhibit more interesting global behaviour. Before we start to give a brief review of the global/asymptotic behaviour, we first give a typical example of global non-scattering solution, i.e. the Talenti-Aubin solution, also called a ground state. 
\[
 W(x) = \left(\frac{1}{3} + |x|^2\right)^{-1/2}. 
\]
Indeed, $W(x)$ solves the elliptic equation $-\Delta u = |u|^4 u$. It is clear that all solutions to this elliptic equation are also solutions to (CP1) independent of time. They are called the stationary solutions to (CP1). Among all stationary solutions (not necessarily radially symmetric) to (CP1), $W(x)$ comes with the smallest energy. This is why we call $W$ a ground state. In the radial case, all non-trivial finite-energy radial stationary solutions can be given by the dilations of $W$, up to a sign, i.e. 
\begin{align*}
 &\{\pm W_\lambda: \lambda\in \Rm^+\};& & W_\lambda \doteq \frac{1}{\lambda^{1/2}} W\left(\frac{x}{\lambda}\right). 
\end{align*}
All these dilations come with the same energy. Thus they are all called ground states. 

\paragraph{Finite time blow-up}  A classical local theory implies that if $u$ blows up at time $T_+ \in \Rm^+$, then  
\[
 \|u\|_{L^5 L^{10} ([0,T_+)\times \Rm^3)} = + \infty. 
\]
Indeed there are two types of finite time blow-up solutions. Type I blow-up solutions satisfy 
 \[
   \lim_{t\rightarrow T_+} \|(u,u_t)\|_{\dot{H}^1\times L^2} = + \infty. 
  \]
 We may construct such a solution in the following way: we start by an explicit solution of (CP1)
 \[
   u(x,t) = \left(\frac{3}{4}\right)^{1/4} (T_+ -t)^{-1/2}, 
  \]
 which blows up as $t\rightarrow T_+$. A combination of cut-off techniques and the finite speed of propagation then gives a finite-energy type I blow-up solution. 

\paragraph{Soliton resolution} In the contrast, a type II blow-up solution satisfies
 \[
  \limsup_{t\rightarrow T_+} \|(u,u_t)\|_{\dot{H}^1\times L^2} < +\infty. 
 \]
 The asymptotic behaviour of type II blow-up solutions and global solutions can be described by the following soliton resolution conjecture: As time tends to the blow-up time or infinity, a solution asymptotically decomposes into a sum of decoupled solitary waves, a free wave and a small error term. Here solitary waves are Lorentz transformations of stationary solutions to (CP1). In the radial case, the solitary waves are simply ground states thus we may write the soliton resolution in the following form 
\[
 \vec{u}(t) = \sum_{j=1}^J \zeta_j (W_{\lambda_j(t)},0) + \vec{u}_L(t) + o(1), \qquad t\rightarrow T_+\;\hbox{or}\; +\infty.
\]
Here $\vec{u}=(u,u_t)$; $\zeta_j \in \{+1,-1\}$ are signs; $u_L$ is a free wave; $o(1)$ is a error term, whose norm in the energy space $\dot{H}^1\times L^2$ vanishes as $t$ tends to $T_+$ or $\infty$. The scale functions $\lambda_j(t)$ for type II blow-up solutions satisfy
\begin{align*}
 & \lim_{t\rightarrow T_+} \frac{\lambda_{j+1}(t)}{\lambda_j(t)} = 0, \quad j=1,2,\cdots,J-1;& & \lim_{t\rightarrow T_+} \frac{\lambda_1(t)}{T_+ -t} = 0. 
\end{align*}
Similarly scale functions for global solutions satisfy 
\begin{align*}
 & \lim_{t\rightarrow +\infty} \frac{\lambda_{j+1}(t)}{\lambda_j(t)} = 0, \quad j=1,2,\cdots,J-1;& & \lim_{t\rightarrow +\infty} \frac{\lambda_1(t)}{t} = 0. 
\end{align*}
Although soliton resolution conjecture is still an open problem in the non-radial case (see Duyckaerts-Jia-Kenig \cite{djknonradial} for a partial result), it has been verified in the radial case in the past fifteen years. Duyckaerts-Kenig-Merle \cite{se} gave the first proof of the 3-dimensional case via a combination of profile decomposition and channel of energy method. Then Duyckaerts-Kenig-Merle \cite{oddhigh}, Duyckaerts-Kenig-Martel-Merle \cite{soliton4d}, Collot-Duyckaerts-Kenig-Merle \cite{soliton6d} and Jendrej-Lawrie \cite{anothersoliton} verified the soliton resolution conjecture in higher dimensions for radial data.   Recently the author \cite{dynamics3d} gives another proof of this conjecture in 3D radial case and discusses some further quantitative properties of the soliton resolution. In addition, soliton resolution for co-rotational wave maps has also been verified by Jendrej-Lawrie \cite{solitonwavemap2}. For the convenience of the readers, we copy below the soliton resolution theorem in the 3D radial case given by Duyckaerts-Kenig-Merle \cite{se}.

\begin{theorem} \label{soliton resolution thm}
 Let $u$ be a radial solution of (CP1) and $T_+ = T_+(u)$ be the right endpoint of this maximal interval of existence. Then one the following holds:
 \begin{itemize}
  \item {\bf Type I blow-up}: $T_+<\infty$ and 
  \[
   \lim_{t\rightarrow T_+} \|(u(t),u_t(t))\|_{\dot{H}^1\times L^2} = +\infty. 
  \]
  \item {\bf Type II blow-up}: $T_+<\infty$ and there exist $(v_0,v_1) \in \dot{H}^1\times L^2$, an integer $J\geq 1$, and for all $j \in \{1,2,\cdots,J\}$, a sign $\zeta_j\in \{\pm 1\}$, and a positive function $\lambda_j(t)$ defined for $t$ close to $T_+$ such that 
  \begin{align*}
   \lambda_1(t) \ll \lambda_2(t) \ll \cdots \ll \lambda_J(t) \ll T_+ - t\quad \hbox{as} \quad t\rightarrow T_+; \\
   \lim_{t\rightarrow T_+} \left\|(u(t),\partial_t u(t)) - \left(v_0 + \sum_{j=1}^J \frac{\zeta_j}{\lambda_j(t)} W\left(\frac{x}{\lambda_j(t)}\right),v_1\right) \right\|_{\dot{H}^1\times L^2} =0.
  \end{align*}
  \item {\bf Global solution}: $T_+ = +\infty$ and there exist a solution $v_L$ of the linear wave equation, an integer $J \geq 0$, and for all $j \in \{1,2,\cdots,J\}$, a sign $\zeta_j \in \{\pm 1\}$, and a positive function $\lambda_j(t)$ defined for large $t$ such that 
   \begin{align*}
   \lambda_1(t) \ll \lambda_2(t) \ll \cdots \ll \lambda_J(t) \ll t\quad \hbox{as} \quad t\rightarrow +\infty; \\
   \lim_{t\rightarrow +\infty} \left\|(u(t),\partial_t u(t)) - \left(v_L(t) + \sum_{j=1}^J \frac{\zeta_j}{\lambda_j(t)} W\left(\frac{x}{\lambda_j(t)}\right),\partial_t v_L(t)\right) \right\|_{\dot{H}^1\times L^2} =0.
  \end{align*}
 \end{itemize}
\end{theorem} 

\paragraph{Number of bubbles} If the soliton resolution of a solution comes with $J$ solitary waves, then we call it a $J$-bubble solution. A scattering solution can be viewed as a $0$-bubble solution as time tends to infinity. 

\subsection{Main topic and result} 

Although the soliton resolution conjecture has been verified in the radial case, a natural question still remains to be answered, i.e. can we find an example of soliton resolution for each combination of bubble number and/or signs? To be more precise, given a positive integer $J$ and a sequence of signs $\zeta_1, \zeta_2, \cdots, \zeta_J$, does it exist a radial global solution (or type II blow-up solution) to (CP1), such that the following soliton resolution holds as time tends to infinity (or blow-up time)?
\[
 \vec{u}(t) \approx \sum_{j=1}^J \zeta_j (W_{\lambda_j(t)},0) + \vec{u}_L(t) + o(1). 
\]
Let us make a brief review on relevant results given in previous works. 

\paragraph{Type II blow-up solutions} The first type II blow-up solution was constructed by Krieger-Schlag-Tataru \cite{slowblowup1}, then by Krieger-Schlag \cite{slowblowup2} and Donninger-Huang-Krieger-Schlag \cite{moreexamples}. All these examples come with a single soliton, but with different choices of scale functions $\lambda_1(t)$. Please note that similar type II blow-up solutions can also be constructed in higher dimensions. Please see Hillairet-Rapha\"{e}l \cite{4dtypeII} and Jendrej \cite{5dtypeII}, for examples. 

\paragraph{Global solutions} The ground states are clearly non-scattering global solutions to (CP1). In addition, Donninger-Krieger \cite{nonscaglobal1} proved that one-bubble global solution exists with a scale function behaving like $\lambda_1(t) \simeq t^{\mu}$ for any sufficiently small parameter $|\mu| \ll 1$. 

\paragraph{Main result} In summary only one-bubble examples are previously known in the 3D radial case. In this work we prove that this is the general rule, i.e. soliton resolution with two or more bubbles does not exist at all in the 3D radial case. Now we introduce the main result of this work:

\begin{theorem} \label{thm main}
 There does not exist any radial global solution or type II blow-up solution to (CP1) with two or more bubbles in its soliton resolution. In other words, if $u$ is a radial solution defined for all time $t\geq 0$, then exactly one of the following holds 
 \begin{itemize}
  \item {\bf Scattering:} there exists a free wave $u_L$, such that 
  \[
   \lim_{t\rightarrow +\infty} \|\vec{u}(t) - \vec{u}_L(t)\|_{\dot{H}^1\times L^2} = 0. 
  \]
  \item {\bf One-bubble global solution:} there exists a finite-energy free wave $u_L$, a sign $\zeta \in \{+1,-1\}$ and a scale function $\lambda(t)>0$ such that 
  \begin{align*}
  & \lim_{t\rightarrow +\infty} \left\|\vec{u}(t) - \vec{u}_L(t) - \frac{\zeta}{\lambda(t)^{1/2}} \left(W\left(\frac{x}{\lambda(t)}\right),0\right)\right\|_{\dot{H}^1\times L^2(\Rm^3)} = 0;& &\lim_{t\rightarrow +\infty} \frac{\lambda(t)}{t} = 0.   
  \end{align*}
 \end{itemize}
 Similarly if a radial solution $u$ to (CP1) blows up at a finite time $T_+ > 0$, then exactly one of the following holds 
 \begin{itemize}
  \item {\bf Type I blow-up:} the solution $u$ blows up in the manner of type I, i.e. 
  \[
   \lim_{t\rightarrow T_+} \|\vec{u}(t)\|_{\dot{H}^1\times L^2(\Rm^3)} = +\infty. 
  \]
  \item {\bf One-bubble type II blow-up:} there exists a finite-energy free wave $u_L$, a sign $\zeta \in \{+1,-1\}$ and a scale function $\lambda(t)>0$ such that 
    \begin{align*}
  & \lim_{t\rightarrow T_+} \left\|\vec{u}(t) - \vec{u}_L(t) - \frac{\zeta}{\lambda(t)^{1/2}} \left(W\left(\frac{x}{\lambda(t)}\right),0\right)\right\|_{\dot{H}^1\times L^2(\Rm^3)} = 0;& &\lim_{t\rightarrow T_+} \frac{\lambda(t)}{T_+ -t} = 0.   
  \end{align*}
 \end{itemize}
 \end{theorem}

Please note that we substitute $(v_0,v_1) \in \dot{H}^1\times L^2$ (as given in Theorem \ref{soliton resolution thm}) by a linear free wave $u_L$ here in the type II blow-up case, for the reason of consistence. It clearly does not make any difference since $\vec{u}_L(t) \rightarrow \vec{u}_L(T_+)$ in the energy space $\dot{H}^1\times L^2$ as $t\rightarrow T_+$. 

\begin{remark}
 Examples of all four cases in Theorem \ref{thm main} are previously known to exist. As a result, Theorem \ref{thm main} finally gives a complete classification of the asymptotic behaviour of any radial solution to (CP1). This is the first complete classification result in the area of soliton resolution for the focusing, energy-critical wave equation, as far as the author knows. 
\end{remark}
  
\begin{remark}
 Multiple bubble solutions to (CP1) do exist in non-radial case. Indeed, if the radially symmetric assumption is removed, then one may construct a type II blow-up solution with any number of solitary waves by combining several type II blow-up solutions with different blow-up points in the space but the same blow-up time, thanks to the finite speed of wave propagation. Furthermore, type II blow-up solution with multiple bubbles shrinking to a single blow-up point but along different directions has also been constructed recently by Kadar \cite{difflines}. In 5-dimensional case, non-radial global solutions with two or more bubbles have also been constructed by Martel-Merle \cite{5dsolitonexample1, 5dinelasticity} and Yuan \cite{5dsolitonexample2}. 
\end{remark} 

\begin{remark} 
 Radial multiple bubble solutions still exist if we consider the energy critical wave equation $\square u = |u|^\frac{4}{d-2}$ in a high-dimensional space $\Rm^d$. For example, two-bubble radial solutions have been constructed by Jendrej \cite{twobubble6d} for $d=6$. However, the author conjectures that given a dimension $d \geq 3$, there exists a positive integer $N= N(d)$, such that the soliton resolution of radial solutions to the energy-critical wave equation $\square u = |u|^\frac{4}{d-2}$ can never come with more than $N$ solitons. Our main theorem verifies that $N(3) = 1$. 
\end{remark}

\begin{remark}
 If we consider a special case with zero radiation $u_L = 0$ in the soliton resolution, then it has been proved by Jendrej \cite{nonexistence2bubble} that solutions with two bubbles of different signs do not exist in the radial case for any dimension $d\geq 3$. 
\end{remark}

\begin{remark}
 Recently soliton resolution with any number of bubbles was constructed for the co-rotational wave maps by Krieger-Palacios \cite{wavemaptower1} and Hwang-Kim \cite{wavemaptower2}. Please note that the bubbles come with alternative signs in the soliton resolution given by both these two works. 
\end{remark}

\subsection{General idea} 

Now we describe the general idea of this work. Generally speaking, we investigate the relationship between the bubbles and radiation of a solution. This idea dates back to Duyckaerts-Kenig-Merle's proof of the soliton resolution conjecture, and the channel of energy method they deployed. In fact, their proof given in \cite{se} utilized an important fact that any radial solution of (CP1) other than zero or ground states must come with nonzero radiation outside the main light cone $|x|=|t|$, i.e. 
\[
 \sum_{\pm} \lim_{t\rightarrow \pm \infty} \int_{|x|>|t|} |\nabla_{t,x} u(x,t)|^2 {\rm d} x > 0. 
\]
Here $\nabla_{t,x} = (\partial_t, \nabla_x)$. 

\paragraph{Radiation fields} The theory of radiation fields may help us further investigate the asymptotic behaviour of solutions to the wave equations as time tends to infinity. The classic theory of radiation fields applies to the free waves. For simplicity we focus on the 3D radially symmetric case. Given any finite-energy radial free wave $u$, there exist two functions $G_\pm \in L^2(\Rm)$ such that 
\begin{align*}
 \lim_{t\rightarrow \pm \infty} \int_0^\infty \left|r u_t(r,t) - G_\pm(r\mp t)\right|^2 {\rm d} r = 0; \\
 \lim_{t\rightarrow \pm \infty} \int_0^\infty \left|r u_r(r,t) \pm G_\pm(r\mp t)\right|^2 {\rm d} r = 0.
\end{align*}
This gives good approximation of the gradient $(u_t, \nabla u)$ in the energy space. The author calls these functions $G_\pm$ the radiation profiles. In many situations, the asymptotic behaviour of a solution to the nonlinear wave equation is similar to that of a free wave, either in the whole space for a given time direction, or outside some light cone. As a result, similar limits to the ones given above hold for suitable radiation profiles $G_\pm$. In other words, we may also describe the asymptotic behaviour of a nonlinear solution by specifying its corresponding radiation profiles. In particular, the radiation strength of a suitable solution $u$ in the ``energy channel'' $\{{x,t}: |t|+r_1<|x|<|t|+r_2\}$ can be measured by the limits 
\[
  \lim_{t\rightarrow \pm \infty} \int_{|t|+r_1<|x|<|t|+r_2} |\nabla_{t,x} u(x,t)|^2 {\rm d} x 
\]
or equivalently, the integrals 
\[
 \int_{r_1}^{r_2} |G_\pm(s)|^2 {\rm d} s. 
\]

\paragraph{Radiation concentration} Let us assume that the following soliton resolution holds ($J\geq 2$)
\[
 \vec{u}(t) \approx \sum_{j=1}^J \zeta_j \left(W_{\lambda_j(t)}, 0\right) + \vec{u}_L(t) + o(1). 
\]
We focus on the interaction of the smallest two bubbles $\zeta_J W_{\lambda_J (t)}$ and $\zeta_{J-1} W_{\lambda_{J-1}(t)}$, and show that a significant radiation concentration has to happen for at least one of the radiation profiles $G_\pm$ associated to the solution $u$. More precisely, we have
\[
 \sup_{0<r<\lambda_{J-1}(t)} \frac{\lambda_{J-1}(t)}{r} \int_0^r \left(|G_+(s-t)|^2 + |G_-(s+t)|^2 \right) {\rm d} s > \kappa;
\]
as long as $t$ is sufficiently large (or sufficiently close to $T_+$). Here $\kappa > 0$ is a constant determined solely by the bubble number $J$. This, together with the classic theory of maximal functions, gives a contradiction. Intuitively strong concentration can not always happen as we make $t \rightarrow +\infty$ (or $t\rightarrow T_+$ in the type II blow-up case).

\paragraph{Bubble interaction with no dispersion} Now let us briefly explain why a strong radiation concentration has to happen as described above. Indeed, if $u$ were a $J$-bubble solution with almost no radiation in the channel $\Psi = \{(x,t): |t|<|x|<|t|+\lambda_{J-1}(t)\}$, then we might linearize the wave equation (CP1) near the approximated solution (given by the soliton resolution)
\[
 S_\ast = \sum_{j=1}^J \zeta_j W_{\lambda_j} (x) + v_L,
\]
where $v_L$ is the asymptotically equivalent free wave of $u$ outside the main light cone, as defined in Subsection \ref{sec: asymptotically equivalent solutions}; and deduce that the error $w_\ast = u - S_\ast$ satisfies the wave equation 
\begin{align*}
 \square w_\ast = F(u) - \sum_{j=1}^J \zeta_j F(W_{\lambda_j}) & = 5  W_{\lambda_J}^4 \left(w_\ast + \zeta_{J-1} W_{\lambda_{J-1}}\right) + \hbox{lower order terms} \\
 & = 5 W_{\lambda_J}^4 \left(w_\ast + \sqrt{3} \zeta_{J-1} \lambda_{J-1}^{-1/2}\right) + \hbox{lower order terms}. 
\end{align*}
Neglecting the lower order terms and solving the wave equation
\begin{equation} \label{approx wave equation intro}
 \square w_\ast = 5 W_{\lambda_J}^4 \left(w_\ast + \sqrt{3} \zeta_{J-1} \lambda_{J-1}^{-1/2}\right),
\end{equation}
we may give a more precise approximation 
\[
 u \approx \sum_{j=1}^J \zeta_j W_{\lambda_j} (x) + v_L + \sqrt{3} \zeta_{J-1} \lambda_{J-1}^{-1/2} \varphi(x/\lambda_J). 
\]
Here $\varphi(x)$ is a well-chosen solution to the linear elliptic equation 
\[ 
 -\Delta \varphi = 5 W^4 \varphi + 5 W^4, 
\]
thus $\sqrt{3} \zeta_{J-1} \lambda_{J-1}^{-1/2} \varphi(x/\lambda_J)$ is exactly a solution of \eqref{approx wave equation intro}. Please note that we make $\varphi$ independent of the time because $\varphi$ is expected to send no radiation. This precise approximation finally gives a contradiction if we consider the behaviour of $u$ near the origin because we may prove that $\varphi(x)$ comes with a strong singularity near the origin.  

\paragraph{Major tools} This work utilize two major tools. The first one is a family of estimates regarding the soliton resolution in term of the Strichartz norm of the asymptotically equivalent free wave $v_L$. Roughly speaking, the following estimates hold for suitable $J$-bubble solution $u$ defined in the exterior region $\Omega_0 = \{(x,t): |x|>|t|\}$
\begin{align*} 
  \left\|\vec{u}(0) -\sum_{j=1}^J \zeta_j\left(W_{\lambda_j},0\right) - \vec{v}_L(0)\right\|_{\dot{H}^1 \times L^2} & \lesssim_J \|\chi_0 v_L\|_{L^5 L^{10}(\Rm \times \Rm^3)};\\
  \frac{\lambda_{j+1}}{\lambda_j} & \lesssim_j \|\chi_0 v_L\|_{L^5 L^{10}(\Rm \times \Rm^3)}, \quad j=1,2,\cdots, J-1. 
\end{align*}
Here $\chi_0$ is the characteristic function of $\Omega_0$; $\zeta_j$ and $\lambda_j$ are signs/scales in the soliton resolution. More details can be found in Section 3. 

Another important tool is a family of refined Strichartz estimates for free waves whose radiation profiles/initial data are supported away from the origin, which are given in Subsection \ref{sec: refined Strichartz}. These refined Strichartz estimates imply that the influence of larger bubbles and the radiation part can be neglected when we consider the interaction of two smallest bubbles in suitable situations. 

\subsection{Structure of this work} 
This work is organized as follows: In Section 2 we first introduce some notations, basic conceptions and preliminary results, including the exterior solutions, radiation fields and theory, asymptotically equivalent solutions, as well as several refined Strichartz estimates. We then make a review of the soliton resolution estimates given by \cite{dynamics3d} in Section 3. Next we discuss the linear elliptic equation mentioned above in Section 4 and give a few approximations of $J$-bubble exterior solutions with fairly weak radiation concentration in Section 5, if this kind of solution existed. Finally in Section 6 we prove the radiation concentration property of $J$-bubble solutions and finish the proof of the main theorem. 

\section{Preliminary results} 

In this section we make a brief review of several previously known theories and results, which will be used in subsequent sections. Let us start by a few notations. 

\paragraph{Implicit constants} In this article we use the notation $A \lesssim B$ if there exists a constant $c$ such that the inequality $A \leq c B$ holds. In addition, we may add subscript(s) to indicate that the implicit constant $c$ mentioned above depends on the subscript(s) but nothing else. In particular, the notation $\lesssim_1$ implies that the constant $c$ is actually an absolute constant. The notations $\gtrsim$ and $\simeq$ can be understood in the same way. Similarly we use the notation $c(\cdot)$, where the dot represents one or more parameter(s), to represent a positive constant determined by the parameter(s) listed but nothing else. In particular, $c(1)$ means an absolute constant. Please note that the same notation $c(\cdot)$ may represent different constants at different places, even if the parameters are exactly the same.  

\paragraph{Box notation} For convenience we use the notation $\square = \partial_t^2 - \Delta$ when necessary in this work.  

\paragraph{Nonlinearity and radial functions} We use the notation $F(u) = |u|^4 u$ throughout this work, unless specified otherwise. If $u$ is a radial solution, then we use the notation $u(r,t)$ to represent the value $u(x,t)$ with $|x|=r$. 

\paragraph{Space norms} We need to consider the restriction of radial $\dot{H}^1$ functions outside a ball of radius $R>0$ when we discuss the conception of exterior solutions. For convenience we let $\mathcal{H}(R)$ be the space of the restrictions of radial functions $(u_0,u_1) \in \dot{H}^1\times L^2$ to the region $\{x: |x|>R\}$, equipped with the norm
\[
 \|(u_0,u_1)\|_{\mathcal{H}(R)} = \left(\int_{|x|>R} \left(|\nabla u_0(x)|^2 + |u_1(x)|^2\right){\rm d} x\right)^{1/2}.
\]
In particular, $\mathcal{H} = \mathcal{H}(0)$ is exactly the Hilbert space of radial $\dot{H}^1\times L^2$ functions.  

\paragraph{Channel-like regions} Throughout this work we use the following notations for the channel-like regions
\begin{align*}
 \Omega_R & = \left\{(x,t)\in \Rm^3 \times \Rm: |x|>|t|+R\right\}, & & R \geq 0; \\
 \Omega_{R_1,r_2} & = \left\{(x,t)\in \Rm^3 \times \Rm: |t|+R_1<|x|<|t|+R_2\right\}, & & 0\leq R_1 < R_2;
\end{align*}
and let $\chi_{R}$, $\chi_{R_1,R_2}$ be their corresponding characteristic functions defined in $\Rm^3 \times \Rm$. In addition, if $\Psi \subset \Rm^3\times \Rm$, then the notation $\chi_{\Psi}$ represents the characteristic function of $\Psi$. 

\paragraph{Strichartz norms} We define $Y$ norm to be the $L^5 L^{10}$ Strichartz norm. For example, if $J$ is a time interval, then 
\[
 \|u\|_{Y(J)} = \|u\|_{L^5 L^{10}(J \times \Rm^3)} = \left(\int_J \left(\int_{\Rm^3} |u(x,t)|^{10} {\rm d} x\right)^{1/2} {\rm d} t\right)^{1/5}. 
\]
We may also combine $Y$ norm with the characteristic function $\chi_{R}$ defined above and write 
\[
 \|\chi_R u\|_{Y(J)} = \left(\int_J \left(\int_{|x|>|t|+R} |u(x,t)|^{10} {\rm d} x\right)^{1/2} {\rm d} t\right)^{1/5}.
\]
Please note that $\|\chi_R u\|_{Y(J)}$ is meaningful even if $u$ is only defined in the exterior region $\Omega_R$ but not necessarily defined in the whole space-time $\Rm^3 \times \Rm$. 

\subsection{Exterior solutions}

In order to focus on the radiation property of $u$ in some exterior region $\Omega_R$, and to avoid (possibly) complicated behaviour of solutions inside the light cone $|x|=|t|+R$, we shall adopt the conception of exterior solutions given by Duyckaerts-Kenig-Merle \cite{oddtool}. We start by discussing exterior solutions to the linear wave equations. Let $u, F$ be functions both defined in the region 
\[
 \Omega =  \{(x,t): |x|>|t|+R,\, t\in (-T_1,T_2)\}\subseteq \Omega_R, \qquad T_1,T_2 \in \Rm^+\cup\{+\infty\}. 
\]
We define the exterior solution $u$ to the following linear wave equation 
\[
  \left\{\begin{array}{l} \partial_t^2 u - \Delta u = F(x,t), \qquad (x,t)\in \Omega;\\ (u,u_t)|_{t=0} = (u_0,u_1) \in \mathcal{H}; \end{array}\right.
\]
where $F$ satisfies $\|\chi_R F\|_{L^1 L^2(J \times \Rm^3)} < +\infty$ for any bounded closed time interval $J \subset (-T_1,T_2)$, by
 \begin{equation} \label{def of exterior sol}
  u = \mathbf{S}_L (u_0,u_1) + \int_0^t \frac{\sin (t-t')\sqrt{-\Delta}}{\sqrt{-\Delta}} [\chi_R(\cdot,t') F(\cdot,t')] {\rm d} t', \qquad (x,t) \in \Omega.
 \end{equation}
Here $\mathbf{S}_L(u_0,u_1)$ is the linear propagation of initial data $(u_0,u_1)$, i.e. the solution to the homogeneous linear wave equation with initial data $(u_0,u_1)$. In other words, $u$ is exactly the restriction of the solution $\tilde{u}$, which solves the following classic linear wave equation, to the exterior region $\Omega$
\[
 \left\{\begin{array}{l} \partial_t^2 \tilde{u} - \Delta \tilde{u} = \chi_R(x,t) F(x,t), \qquad (x,t)\in \Rm \times (-T_1,T_2);\\ (\tilde{u},\tilde{u}_t)|_{t=0} = (u_0,u_1) \in \mathcal{H}.\end{array}\right.
\]
Please note that the finite speed of wave propagation implies that the values of initial data in the ball $\{x: |x|\leq R\}$ are actually irrelevant, thus it suffices to specify the initial data $(u_0,u_1)$ in the space $\mathcal{H}(R)$. We may define an exterior solution $u$ to nonlinear wave equations in a similar way. For instance, we say that a function $u$ defined in $\Omega$ is a solution to
 \[
  \left\{\begin{array}{l} \partial_t^2 u - \Delta u = F(u), \qquad (x,t)\in \Omega;\\ (u,u_t)|_{t=0} = (u_0,u_1) \end{array}\right.
 \]
if and only if the inequality $\|\chi_R u\|_{Y(J)} < +\infty$ holds for any bounded closed time interval $J \subset (-T_1,T_2)$, which also implies that $\|\chi_R F(u)\|_{L^1L^2(J\times \Rm^3)} < +\infty$, and the identity \eqref{def of exterior sol} holds with $F(x,t) = F(u(x,t))$.  
 
\paragraph{Local theory} A combination of the Strichartz estimates (see \cite{strichartz} for example) 
\[
 \sup_{t} \|\vec{u}(t)\|_{\mathcal{H}} + \|u\|_{L^5 L^{10}} \lesssim_1 \|\vec{u}(0)\|_{\mathcal{H}} + \|\square u\|_{L^1 L^2},
\]
finite speed of wave propagation and a fixed-point argument immediately leads to the local well-posedness theory, small data scattering theory and perturbation theory(continuous dependence of solutions on the initial data) of exterior solutions. The argument is almost the same as the corresponding argument in the whole space $\Rm^3$. More details of this argument in the whole space can be found in \cite{loc1, ls} for local well-posedness and in \cite{kenig, shen2} for perturbation theory.  

\subsection{Radiation fields} \label{sec: radiation fields}

The theory of radiation fields plays an important role in the discussion of the asymptotic behaviour of solutions to wave equations. It dates back to Friedlander's works \cite{radiation1, radiation2} more than half a century ago. The following version of statement comes from Duyckaerts-Kenig-Merle \cite{dkm3}.

\begin{theorem}[Radiation field] \label{radiation}
Assume that $d\geq 3$ and let $u$ be a solution to the free wave equation $\partial_t^2 u - \Delta u = 0$ with initial data $(u_0,u_1) \in \dot{H}^1 \times L^2(\Rm^d)$. Then
\[
 \lim_{t\rightarrow \pm \infty} \int_{\Rm^d} \left(|\nabla u(x,t)|^2 - |u_r(x,t)|^2 + \frac{|u(x,t)|^2}{|x|^2}\right) {\rm d}x = 0
\]
 and there exist two functions $G_\pm \in L^2(\Rm \times \mathbb{S}^{d-1})$ such that
\begin{align*}
 \lim_{t\rightarrow \pm\infty} \int_0^\infty \int_{\mathbb{S}^{d-1}} \left|r^{\frac{d-1}{2}} \partial_t u(r\theta, t) - G_\pm (r\mp t, \theta)\right|^2 {\rm d}\theta {\rm d}r &= 0;\\
 \lim_{t\rightarrow \pm\infty} \int_0^\infty \int_{\mathbb{S}^{d-1}} \left|r^{\frac{d-1}{2}} \partial_r u(r\theta, t) \pm G_\pm (r\mp t, \theta)\right|^2 {\rm d}\theta {\rm d} r & = 0.
\end{align*}
In addition, the maps $(u_0,u_1) \rightarrow \sqrt{2} G_\pm$ are bijective isometries from $\dot{H}^1 \times L^2(\Rm^d)$ to $L^2 (\Rm \times \mathbb{S}^{d-1})$. 
\end{theorem}

In this work the author calls the functions $G_\pm$ the radiation profiles of the free wave $u$. In addition, the radiation profiles of $(u_0,u_1)\in \dot{H}^1\times L^2$ are defined to be the radiation profiles of the corresponding free wave with initial data $(u_0,u_1)$. It is not difficult to see that $u$ is radial if and only if the radiation profile is independent of the angle $\theta$. In this work we frequently utilize the following explicit formula in the 3D radial case, which give the free wave $u$ and the radiation profile $G_+$ in the positive time direction in term of the radiation profile $G_-$ in the negative time direction. 
\begin{align}
 & u(r,t) = \frac{1}{r} \int_{t-r}^{t+r} G_-(s) {\rm d}s; & & G_+(s) = - G_-(-s),\quad s\in \Rm. \label{basic identity radiation free wave}
\end{align}
Similar formula for other dimensions and non-radial case can be found in \cite{newradiation, shenradiation}. The symmetry between $G_\pm$ given above also implies that an arbitrary combination of $G_\pm \in L^2(\Rm^+)$ uniquely determine a radial free wave. We may also write initial data $(u_0,u_1)$ in term of the initial data $(u_0,u_1)$
\begin{align} \label{initial data by radiation profile} 
 &u_0(r) =  \frac{1}{r} \int_{-r}^{r} G_-(s) {\rm d}s; & & u_1(r) = \frac{G_-(r)-G_-(-r)}{r}.
\end{align}
An integration by parts then gives us another useful formula 
\begin{equation} \label{radiation residue identity}
  \|(u_0,u_1)\|_{\mathcal{H}(R)}^2 = 8\pi \|G_-\|_{L^2(\{s: |s|>R\})}^2 + 4\pi R |u_0(R)|^2. 
\end{equation} 
Next we consider the radiation profile of initial data $(v_0,v_1) \in \mathcal{H}(R)$ for some radius $R>0$. It is natural to define its radiation profile $G_\pm$ in the following way: we pick up radial initial data $(u_0,u_1)\in \dot{H}^1\times L^2$ such that the restriction of $(u_0,u_1)$ in the exterior region $\{x: |x|>R\}$ is exactly $(v_0,v_1)$ and define $G_\pm$ to be the corresponding radiation profile of $(u_0,u_1)$. Although the choice of $(u_0,u_1)$ is NOT unique, we may uniquely determine the value of $G_\pm(s)$ for $|s| > R$, by the finite speed of propagation, as well as the value of 
\[
 \int_{-R}^R G_-(s) {\rm d} s, 
\]
by the explicit formula \eqref{initial data by radiation profile}. Conversely, if two radiation profiles $G_-(s)$ and $\tilde{G}_-(s)$ satisfy
\begin{align*}
 &G_-(s) = \tilde{G}_-(s), \quad |s|>R;& & \int_{-R}^R G_-(s) {\rm d} s = \int_{-R}^R \tilde{G}_-(s) {\rm d} s;
\end{align*}
then the corresponding free wave coincide in the exterior region $\Omega_R$. In summary, the map from initial data $(u_0,u_1)\in \mathcal{H}(R)$ to the corresponding radiation profiles
\[
 (u_0,u_1)\longrightarrow \left(\sqrt{8\pi}G_\pm (s), \frac{2\sqrt{\pi}}{R^{1/2}}\int_{-R}^R G_\pm(s) {\rm d} s\right)
\]
 is an isometric homeomorphism from $\mathcal{H}(R)$ to $L^2(\{s: |s|>R\}) \oplus \Rm$. The isometric property follows from the identities \eqref{initial data by radiation profile} and \eqref{radiation residue identity}.

Finally we may also consider radiation fields and profiles for suitable solutions to inhomogeneous/nonlinear wave equations. 

\begin{lemma} [Radiation fields of inhomogeneous equation] \label{scatter profile of nonlinear solution}
 Assume that $R\geq 0$. Let $u$ be a radial exterior solution to the wave equation
 \[
  \left\{\begin{array}{ll} \partial_t^2 u - \Delta u = F(t,x); & (x,t)\in \Omega_R; \\
  (u,u_t)|_{t=0} = (u_0,u_1) \in \dot{H}^1\times L^2. & \end{array} \right.
 \]
 If $F$ satisfies $\|\chi_R F\|_{L^1 L^2(\Rm\times \Rm^3)}< +\infty$, then there exist unique radiation profiles $G_\pm \in L^2([R,+\infty))$ such that 
 \begin{align*}
  \lim_{t\rightarrow +\infty} \int_{R+t}^\infty \left(\left|G_+(r-t) - r u_t (r, t)\right|^2 + \left|G_+(r-t) + r u_r (r, t)\right|^2\right) {\rm d}r & = 0; \\
  \lim_{t\rightarrow -\infty} \int_{R-t}^\infty \left(\left|G_-(r+t) - r u_t(r,t)\right|^2 +  \left|G_-(r+t) - r u_r(r,t)\right|^2\right) {\rm d} r & = 0. 
 \end{align*}
 In addition, the following estimates hold for $G_\pm$ given above and the corresponding radiation profiles $G_{0,\pm}$ of the initial data $(u_0,u_1)$:
 \begin{align*}
 4\sqrt{\pi} \|G_- - G_{0,-}\|_{L^2([R,R'])} & \leq \|\chi_{R,R'} F\|_{L^1 L^2((-\infty,0]\times \Rm^3)}, & & R'>R; \\
  4\sqrt{\pi}  \|G_+ - G_{0,+}\|_{L^2([R,R'])} & \leq \|\chi_{R,R'} F\|_{L^1 L^2([0,+\infty)\times \Rm^3)}, & & R'>R.
 \end{align*}
\end{lemma}
Please refer to Section 2 (Lemma 2.5 and Remark 2.6) of the author's previous work \cite{dynamics3d} for the proof of this lemma. In fact we may give an explicit formula
 \[
  G_+ (s) - G_{0,+} (s) = \frac{1}{2} \int_{0}^\infty (s+t) F(s+t,t) {\rm d} t. 
 \]
\paragraph{Nonlinear equations} Please note that Lemma \ref{scatter profile of nonlinear solution} applies to exterior solutions $u$ to (CP1) defined in $\Omega_R$, as long as the inequality $\|\chi_R u\|_{Y(\Rm)} < +\infty$ holds, because this assumption guarantees that the inequality $\|\chi_R F(u)\|_{L^1 L^2(\Rm \times \Rm^3)} < +\infty$ holds. In this case the corresponding radiation profiles $G_\pm \in L^2([R,+\infty))$ given in Lemma \ref{scatter profile of nonlinear solution} is called the (nonlinear) radiation profile of $u$. 

\begin{remark}
 Whenever we talk about a radiation profile in subsequent sections without specifying whether it is the radiation profile in the positive or negative direction, we refer to the radiation profile in the negative time direction. 
\end{remark}

\subsection{Asymptotically equivalent solutions} \label{sec: asymptotically equivalent solutions}

Assume that $u, v \in \mathcal{C}(\Rm; \dot{H}^1\times L^2)$ and $R\geq 0$. We say that $u$ and $v$ are $R$-weakly asymptotically equivalent if and only if the following limit holds
\[
 \lim_{t\rightarrow \pm \infty} \int_{|x|>R+|t|} |\nabla_{t,x} (u-v)|^2 {\rm d} x = 0.
\]
In particular, we say that $u$ and $v$ are asymptotically equivalent to each other if $R=0$. Because the integral above only involves the values of $u, v$ in the exterior region $\Omega_R$, the definition above also applies to suitable exterior solutions $u$ and $v$ defined in $\Omega_R$ only.  

\paragraph{Radiation part} If a free wave $v_L$ is asymptotically equivalent to a radial exterior solution $u$ of (CP1), then we call $v_L$ the radiation part of $u$ (outside the light cone). It was prove in \cite{ecarbitrary} that an exterior solution $u$ is asymptotically equivalent to some free wave in $\Omega_0$ if and only if $\|\chi_0 u\|_{Y(\Rm)} < +\infty$. The sufficiency of this condition is a direct consequence of Lemma \ref{scatter profile of nonlinear solution}. Indeed, if $\|\chi_0 u\|_{Y(\Rm)}<+\infty$, then we may determine its (nonlinear) radiation profile $G_\pm \in L^2(\Rm^+)$ by Lemma \ref{scatter profile of nonlinear solution}, and then construct a free wave with the same radiation profiles for $s>0$, which is the desired asymptotically equivalent free wave. To see why the condition $\|\chi_0 u\|_{Y(\Rm)} < +\infty$ is also necessary, please refer to \cite{ecarbitrary}. Please note that this conception of radiation part is different from the radiation part $u_L$ in a soliton resolution at the blow-up time $T_+$ or $+\infty$, as described in the introduction section of this article. 

\paragraph{Non-radiative solutions} A ($R$-weakly) non-radiative solution is a solution $u$ to the free wave equation, or the nonlinear wave equation (CP1), or any other related wave equation such that 
\[
 \lim_{t\rightarrow \pm \infty} \int_{|x|>|t|+R} |\nabla_{t,x} u(x,t)|^2 {\rm d} x = 0.
\]
In other words, a solution $u$ is ($R$-weakly) non-radiative if and only if it is ($R$-weakly) asymptotically equivalent to zero. Non-radiative solution is one of  most important topics in the channel of energy method (see \cite{channeleven, tkm1, channel} for example), which plays an important role in the study of nonlinear wave equations in recent years.  

For an example, the ground states $W_\lambda(x)$ are all non-radiative solutions to (CP1). Thus if a free wave $v_L$ is asymptotically equivalent to a solution $u$ to (CP1), then 
\[
 S_\ast(x,t) = \sum_{j=1}^J \zeta_j W_{\lambda_j}(x) + v_L(x,t) 
\]
is also asymptotically equivalent to $u$, for any given positive integer $J$, signs $\zeta_j \in \{+1,-1\}$ and scales $\lambda_j > 0$. 

\begin{remark}
 The standard ground state $W(x)$ in this article is a dilation of (thus slightly different from) the one $(1+|x|^2/3)^{-1/2}$ used in most related works. This choice eliminates the addition constant in the asymptotic behaviour $W_\lambda(r) \approx \lambda^{1/2} r^{-1}$ for large $r$.
\end{remark}

\subsection{Several Strichartz estimates} \label{sec: refined Strichartz} 

In this subsection we prove several refined Strichartz estimates for further use. Most of them can be verified by a straight-forward calculation. The first few lemmata are concerning the Strichartz norm of the ground state in several channel-like regions. 
\begin{lemma} \label{upper bound of W norm}
 Let $0 \leq r_1 <r_2\leq R$ be positive constants, $R\geq 1$ and 
 \[ 
  \Psi = \{(x,t): r_1+|t|<|x|<r_2+|t|, |x|+|t| > R\}.
 \]
Then we have 
 \[
  \|\chi_{\Psi} W\|_{Y(\Rm)} \lesssim_1 (r_2-r_1)^{1/10} R^{-3/5}. 
 \]
 In addition, if $0\leq r_1 < r_2$, then 
 \[
  \|\chi_{r_1,r_2} W\|_{Y(\Rm)} \lesssim_1 (r_2-r_1)^{1/10} \min\left\{r_1^{-3/5},1\right\}. 
 \]
\end{lemma}
\begin{proof}
 The proof follows a straight-forward calculation. 
 \begin{align*}
  \|\chi_{\Psi} W\|_{Y(\Rm)}^5 &\lesssim_1 \int_{\frac{R-r_2}{2}}^\infty \left(\int_{\max\{t+r_1,R-t\}}^{t+r_2} \left(\frac{1}{3} + r^2\right)^{-5} r^2 {\rm d} r\right)^{1/2} {\rm d} t \\
  & \lesssim_1 \int_{\frac{R-r_2}{2}}^R \left(R^{-8} (r_2 - r_1)\right)^{1/2} {\rm d} t + \int_R^\infty \left(\frac{r_2 - r_1}{(t+r_1)^8}\right)^{1/2} {\rm d} t\\
  & \lesssim_1 (r_2-r_1)^{1/2} R^{-3}. 
 \end{align*}
 Here we only consider the positive time direction by time symmetry and use the inequality $\max\{t+r_1,R-t\} \geq R/2$. The second inequalities can be proved in the same manner. On one hand, we have  
 \begin{align*}
  \|\chi_{r_1,r_2} W\|_{Y(\Rm)}^5 &\lesssim_1 \int_{0}^\infty \left(\int_{t+r_1}^{t+r_2} \left(\frac{1}{3} + r^2\right)^{-5} r^2 {\rm d} r\right)^{1/2} {\rm d} t \\
  & \lesssim_1 \int_0^\infty \left(\frac{r_2-r_1}{(t+r_1)^8}\right)^{1/2} {\rm d} t \\
  & \lesssim_1 (r_2-r_1)^{1/2} r_1^{-3}. 
 \end{align*}
 On the other hand, since 
 \[
   \left(\frac{1}{3} + r^2\right)^{-5} r^2 \lesssim_1 \min\left\{1, r^{-8}\right\},
 \]
 we also have
 \begin{align*}
  \|\chi_{r_1,r_2} W\|_{Y(\Rm)}^5 &\lesssim_1 \int_{0}^\infty \left(\int_{t+r_1}^{t+r_2} \left(\frac{1}{3} + r^2\right)^{-5} r^2 {\rm d} r\right)^{1/2} {\rm d} t \\
  & \lesssim_1 \int_0^1 \left(r_2-r_1\right)^{1/2} {\rm d} t + \int_1^\infty \left(\frac{r_2 - r_1}{(t+r_1)^8}\right)^{1/2} {\rm d} t\\
  & \lesssim_1 (r_2-r_1)^{1/2}. 
 \end{align*}
 Combining these two upper bounds, we finish the proof. 
\end{proof}

\begin{lemma} \label{L1L2 channel estimate mixed}
 Let $\lambda > 1$ be a radius
 \begin{itemize}
  \item If $1/2 \leq r_1 < r_2$, then 
  \begin{align*}
  & \left\|\chi_{r_1,r_2} W^4 W_\lambda\right\|_{L^1 L^2} \lesssim_1   \lambda^{-1/2} (r_2-r_1)^{1/2} r_1^{-2}; &
   &\left\|\chi_{r_1,r_2} W^3 W_\lambda^2\right\|_{L^1 L^2}  \lesssim_1 \lambda^{-1} (r_2-r_1)^{1/2} r_1^{-1}; 
  \end{align*} 
  \item If $0\leq r_1 < r_2 \leq 1$, then
  \begin{align*}
  & \left\|\chi_{r_1,r_2} W^4 W_\lambda\right\|_{L^1 L^2}  \lesssim_1 \lambda^{-1/2} (r_2-r_1)^{1/2};  &
   &\left\|\chi_{r_1,r_2} W^3 W_\lambda^2\right\|_{L^1 L^2}  \lesssim_1 \lambda^{-1} (r_2-r_1)^{1/2}. 
  \end{align*} 
  \item If $0\leq r_1 < r_2 \leq \lambda$, then
  \[
   \left\|\chi_{r_1,r_2} W W_\lambda^4\right\|_{L^1 L^2}  \lesssim_1 \lambda^{-1} (r_2-r_1)^{1/2}.
  \]
 \end{itemize}
\end{lemma}
Please note that all the $L^1 L^2$ norms are the abbreviation of $L^1 L^2(\Rm \times \Rm^3)$ in this work, unless specified otherwise. 
\begin{proof}
 The proof is simply a straight forward calculation. We first assume that $1/2 \leq r_1 < r_2$. Then we have
 \begin{align*}
  \left\|\chi_{r_1,r_2} W^4 W_\lambda\right\|_{L^1 L^2} &\lesssim_1 \int_0^\infty \left(\int_{r_1+t}^{r_2+t} \left(\frac{1}{3} + r^2\right)^{-4} \frac{1}{\lambda}\left(\frac{1}{3}+\frac{r^2}{\lambda^2}\right)^{-1}r^2 {\rm d} r \right)^{1/2} {\rm d} t\\
  & \lesssim_1\int_{0}^\infty \left(\int_{r_1+t}^{r_2+t} \frac{1}{\lambda r^6} {\rm d} r \right)^{1/2} {\rm d} t \\
  & \lesssim_1 \int_{0}^\infty \frac{(r_2-r_1)^{1/2}}{\lambda^{1/2} (r_1+t)^3} {\rm d} t \\
  & \lesssim_1 \lambda^{-1/2} (r_2-r_1)^{1/2} r_1^{-2}. 
 \end{align*}
  \begin{align*}
  \left\|\chi_{r_1,r_2} W^3 W_\lambda^2\right\|_{L^1 L^2} &\lesssim_1 \int_0^\infty \left(\int_{r_1+t}^{r_2+t} \left(\frac{1}{3} + r^2\right)^{-3} \frac{1}{\lambda^2}\left(\frac{1}{3}+\frac{r^2}{\lambda^2}\right)^{-2}r^2 {\rm d} r \right)^{1/2} {\rm d} t\\
  & \lesssim_1\int_{0}^\infty \left(\int_{r_1+t}^{r_2+t} \frac{1}{\lambda^2 r^4} {\rm d} r \right)^{1/2} {\rm d} t \\
  & \lesssim_1 \int_{0}^\infty \frac{(r_2-r_1)^{1/2}}{\lambda (r_1+t)^2} {\rm d} t\\
  & \lesssim_1 \lambda^{-1} (r_2-r_1)^{1/2} r_1^{-1}. 
 \end{align*}
 On the other hand, if $0\leq r_1 < r_2 \leq 1$, then 
  \begin{align*}
  \left\|\chi_{r_1,r_2} W^4 W_\lambda\right\|_{L^1 L^2} &\lesssim_1 \int_0^\infty \left(\int_{r_1+t}^{r_2+t} \left(\frac{1}{3} + r^2\right)^{-4} \frac{1}{\lambda}\left(\frac{1}{3}+\frac{r^2}{\lambda^2}\right)^{-1}r^2 {\rm d} r \right)^{1/2} {\rm d} t\\
  & \lesssim_1 \int_{0}^1 \left(\int_{r_1+t}^{r_2+t} \frac{1}{\lambda} {\rm d} r \right)^{1/2} {\rm d} t  + \int_{1}^\infty \left(\int_{r_1+t}^{r_2+t} \frac{1}{\lambda r^6} {\rm d} r \right)^{1/2} {\rm d} t \\
  & \lesssim_1 \int_0^1 \frac{(r_2-r_1)^{1/2}}{\lambda^{1/2}}{\rm d} t + \int_{1}^\infty \frac{(r_2-r_1)^{1/2}}{\lambda^{1/2} (r_1+t)^3} {\rm d} t \\
  & \lesssim_1 \lambda^{-1/2} (r_2-r_1)^{1/2}. 
 \end{align*}
  \begin{align*}
  \left\|\chi_{r_1,r_2} W^3 W_\lambda^2 \right\|_{L^1 L^2} &\lesssim_1 \int_0^\infty \left(\int_{r_1+t}^{r_2+t} \left(\frac{1}{3} + r^2\right)^{-3} \frac{1}{\lambda^2}\left(\frac{1}{3}+\frac{r^2}{\lambda^2}\right)^{-2}r^2 {\rm d} r \right)^{1/2} {\rm d} t\\
  & \lesssim_1 \int_{0}^1 \left(\int_{r_1+t}^{r_2+t} \frac{1}{\lambda^2} {\rm d} r \right)^{1/2} {\rm d} t  + \int_{1}^\infty \left(\int_{r_1+t}^{r_2+t} \frac{1}{\lambda^2 r^4} {\rm d} r \right)^{1/2} {\rm d} t \\
  & \lesssim_1 \int_0^1 \frac{(r_2-r_1)^{1/2}}{\lambda}{\rm d} t + \int_{1}^\infty \frac{(r_2-r_1)^{1/2}}{\lambda (r_1+t)^2} {\rm d} t\\  
  & \lesssim_1 \lambda^{-1} (r_2-r_1)^{1/2}. 
 \end{align*}
 Finally, for $0\leq r_1<r_2\leq \lambda$, we have
  \begin{align*}
  \left\|\chi_{r_1,r_2} W W_\lambda^4\right\|_{L^1 L^2} &\lesssim_1 \int_0^\infty \left(\int_{r_1+t}^{r_2+t} \left(\frac{1}{3} + r^2\right)^{-1} \frac{1}{\lambda^4}\left(\frac{1}{3}+\frac{r^2}{\lambda^2}\right)^{-4}r^2 {\rm d} r \right)^{1/2} {\rm d} t\\
  & \lesssim_1\int_{0}^\lambda \left(\int_{r_1+t}^{r_2+t} \frac{1}{\lambda^4} {\rm d} r \right)^{1/2} {\rm d} t  + \int_\lambda^\infty \left(\int_{r_1+t}^{r_2+t} \frac{\lambda^4}{r^8} {\rm d} r \right)^{1/2} {\rm d} t\\
  & \lesssim_1 \int_{0}^\lambda \frac{(r_2-r_1)^{1/2}}{\lambda^2} {\rm d} t  + \int_\lambda^\infty \frac{\lambda^2 (r_2-r_1)^{1/2}}{(r_1+t)^4} {\rm d} t\\
  & \lesssim_1 \lambda^{-1} (r_2-r_1)^{1/2}. 
 \end{align*}
\end{proof}

\begin{corollary} \label{global interaction bubbles}
 Let $\lambda > 1$ be a radius. Then 
 \begin{align*}
  \left\|W^4 W_\lambda\right\|_{L^1 L^2} +  \left\|W W_\lambda^4\right\|_{L^1 L^2} \lesssim_1 \lambda^{-1/2}. 
 \end{align*}
\end{corollary}
\begin{proof} 
 The estimate of $W^4 W_\lambda$ is a direct consequence of Lemma \ref{L1L2 channel estimate mixed}. 
 \begin{align*}
  \left\|W^4 W_\lambda\right\|_{L^1 L^2} & \leq \left\|\chi_{0,1} W^4 W_\lambda\right\|_{L^1 L^2} + \sum_{k=0}^\infty \left\|\chi_{2^k,2^{k+1}} W^4 W_\lambda\right\|_{L^1 L^2} \\
  & \lesssim_1 \lambda^{-1/2} + \sum_{k=0}^\infty \lambda^{-1/2} 2^{-3k/2} \lesssim_1 \lambda^{-1/2}. 
 \end{align*}
 Next we conduct a direct calculation
 \begin{align*}
  \left\|\chi_\lambda W W_\lambda^4\right\|_{L^1 L^2} & \lesssim_1 \int_0^\infty \left(\int_{\lambda+t}^{\infty} \left(\frac{1}{3} + r^2\right)^{-1} \frac{1}{\lambda^4}\left(\frac{1}{3}+\frac{r^2}{\lambda^2}\right)^{-4}r^2 {\rm d} r \right)^{1/2} {\rm d} t\\
  & \lesssim_1 \int_0^\infty \left(\int_{\lambda+t}^{\infty} \frac{\lambda^4}{r^8} {\rm d} r \right)^{1/2} {\rm d} t\\
  & \lesssim_1 \int_0^\infty \frac{\lambda^2}{(\lambda+t)^{7/2}} {\rm d} t\\
  & \lesssim_1 \lambda^{-1/2}. 
 \end{align*}
 Combining this with Lemma \ref{L1L2 channel estimate mixed}, we obtain 
 \[
  \left\|W W_\lambda^4\right\|_{L^1 L^2} \leq \left\|\chi_{0,\lambda} W W_\lambda^4\right\|_{L^1 L^2} + \left\|\chi_\lambda W W_\lambda^4\right\|_{L^1 L^2} \lesssim_1 \lambda^{-1/2}.
 \]
 This completes the proof. 
\end{proof}

\begin{lemma} \label{estimate W J lambda}
 Assume that $\lambda \geq 2$ and $0\leq r_1 < r_2 \leq \lambda$. Then 
 \[ 
  \left\|\chi_{r_1,r_2} W^4 (W_\lambda - \sqrt{3} \lambda^{-1/2})\right\|_{L^1 L^2} \lesssim_1 (r_2-r_1)^{1/2} \lambda^{-5/2} \ln \lambda.
 \]
\end{lemma}
\begin{proof}
 we observe that
 \[
  \left|W_\lambda - \sqrt{3} \lambda^{-1/2}\right| \lesssim_1 \min\left\{\lambda^{-5/2} r^2, \lambda^{-1/2}\right\}.
 \]
Thus 
\begin{align*}
 \hbox{LHS} & \lesssim_1 \int_0^\infty \left(\int_{r_1+t}^{r_2+t} \left(\frac{1}{3} + r^2\right)^{-4} \left|W_\lambda - \sqrt{3} \lambda^{-1/2}\right|^2 r^2 {\rm d} r\right)^{1/2} {\rm d} t \\
 & \lesssim_1 \int_0^\lambda \left(\int_{r_1+t}^{r_2+t} \lambda^{-5} \left(\frac{1}{3} + r^2\right)^{-1} {\rm d} r\right)^{1/2} {\rm d} t + \int_\lambda^\infty \left(\int_{r_1+t}^{r_2+t} \lambda^{-1} r^{-6} {\rm d} r\right)^{1/2} {\rm d} t\\
 & \lesssim_1 \int_0^\lambda \frac{(r_2-r_1)^{1/2}}{\lambda^{5/2}} \left(\frac{1}{3} + (r_1+t)^2\right)^{-1/2} {\rm d} t + \int_\lambda^\infty \frac{(r_2-r_1)^{1/2}}{\lambda^{1/2} (r_1+t)^3} {\rm d} t\\
 & \lesssim_1 (r_2-r_1)^{1/2} \lambda^{-5/2} \ln \lambda. 
\end{align*}
\end{proof}

The following results are concerning the $L^5 L^{10}$ norm of free waves in the channel-like regions. 
\begin{lemma} \label{lemma csg}
 Let $v$ be a radial free wave with a radiation profile $G$ and $0\leq r_1<r_2\leq R$ be radii. If $G(s) = 0$ for $|s|<R$, then we have 
 \[
  \|\chi_{r_1,r_2} v\|_{Y(\Rm)} \lesssim_1\left(\frac{r_2-r_1}{R}\right)^\frac{1}{10} \|G\|_{L^2(\Rm)}.
 \]
\end{lemma}
\begin{proof}
 We recall the explicit formula 
 \[
   v(r,t) =  \frac{1}{r} \int_{t-r}^{t+r} G(s) {\rm d} s, 
 \]
 which immediately gives the point-wise estimate
 \[
  |v(r,t)| \lesssim_1 r^{-1/2} \|G\|_{L^2(\Rm)}. 
 \]
 In addition, if $|t|<r<R/2$, then the support of $G$ guarantees that $v(r,t) =0$. Thus 
 \[
  |v(r,t)| \lesssim_1 \min\left\{r^{-1/2}, R^{-1/2}\right\} \|G\|_{L^2}, \qquad r>|t|. 
 \]
 A direct calculation then gives 
 \begin{align*}
  \|\chi_{r_1,r_2} v\|_{Y(\Rm)}^5 & \lesssim_1 \int_\Rm \left(\int_{|t|+r_1}^{|t|+r_2} \min\{r^{-5},R^{-5}\}\|G\|_{L^2}^{10} \cdot r^2 {\rm d} r\right)^{1/2} {\rm d} t \\
  & \lesssim_1 \int_{-R}^R \left(\int_{|t|+r_1}^{|t|+r_2} R^{-5} r^2 \|G\|_{L^2}^{10} {\rm d} r\right)^{1/2} {\rm d} t + \int_{|t|>R} \left(\int_{|t|+r_1}^{|t|+r_2} r^{-3} \|G\|_{L^2}^{10} {\rm d} r\right)^{1/2} {\rm d} t\\
  & \lesssim_1 \int_{-R}^R \left[R^{-3} \|G\|_{L^2}^{10} (r_2-r_1)\right]^{1/2} {\rm d} t + \int_{|t|>R} \left(\frac{(r_2-r_1)\|G\|_{L^2}^{10}}{(|t|+r_1)^3}\right)^{1/2} {\rm d} t\\
  & \lesssim_1 \left(\frac{r_2-r_1}{R}\right)^\frac{1}{2} \|G\|_{L^2(\Rm)}^5. 
 \end{align*}
 \end{proof}

\begin{corollary} \label{lemma split initial data}
 Assume that $0<R_0<R_1<R_2<\cdots < R_{m+1}$ is a sequence of positive numbers.  Let $(u_0,u_1) \in \mathcal{H}(R_0)$ be radial initial data with radiation profile $G(s)$. Then the corresponding free wave $u$ satisfies 
 \begin{align*}
  &\left\|\chi_{R_0,R_1} u\right\|_{Y(\Rm)}  \lesssim_1 R_0^{-1/2} \left|\int_{-R_0}^{R_0} G(s) {\rm d} s\right| + \|G\|_{L^2(\{s: R_0<|s|<R_1\})} \\
 & \qquad + \sum_{j=1}^m \left(\frac{R_1-R_0}{R_{j}}\right)^{1/10} \|G\|_{L^2(\{s: R_j <|s|<R_{j+1}\})} + \left(\frac{R_1-R_0}{R_{m+1}}\right)^{1/10} \|G\|_{L^2(\{s: |s|>R_{m+1}\})}. 
 \end{align*}
\end{corollary} 
\begin{proof}
 We may split the linear free wave $u$ into several ones 
 \begin{align*}
  u(r,t) = \frac{1}{r} \int_{-R_0}^{R_0} G(s) {\rm d} s + \sum_{j=0}^{m+1} v_j (r,t), \qquad r>R_0+|t|.
 \end{align*} 
 Here $v_j$ is the free wave whose radiation profile is exactly the restriction of $G(s)$ on the set $\{s: R_j < |s|<R_{j+1}\}$ (or $\{s: |s|>R_{m+1}\}$ for $j=m+1$). A direct calculation then shows that 
 \[
  \left\|\chi_{R_0} |x|^{-1}\right\|_{Y(\Rm)} \lesssim_1 R_0^{-1/2}. 
 \]
 The conclusion then follows from a combination of this upper bound, the classic Strichartz estimate and Lemma \ref{lemma csg}. 
\end{proof}
 
\begin{remark} \label{remark split initial data}
 Let $0 = R_0 < R_1 < R_2 < \cdots < R_{m+1}$ be a sequence and $(u_0,u_1) \in \dot{H}^1 \times L^2$ be radial initial data with radiation profile $G(s)$. Then the same argument as above gives 
  \begin{align*}
  &\left\|\chi_{0,R_1} u\right\|_{Y(\Rm)}  \lesssim_1 \|G\|_{L^2(\{s: 0<|s|<R_1\})} \\
 & \qquad + \sum_{j=1}^m \left(\frac{R_1}{R_{j}}\right)^{1/10} \|G\|_{L^2(\{s: R_j <|s|<R_{j+1}\})} + \left(\frac{R_1}{R_{m+1}}\right)^{1/10} \|G\|_{L^2(\{s: |s|>R_{m+1}\})}. 
 \end{align*}
\end{remark} 
 
The following results give an upper bound of $L^5 L^{10}$ norm for solutions to the linear wave equation with localized data. 
\begin{lemma} \label{lemma csit}
 Assume that $0\leq r_1 < r_2 \leq R$. Let $u_1 \in L^2(\Rm^3)$ be radial function supported in $\{x: |x|>R\}$. Then the free wave $v = \mathbf{S}_L(0,u_1)$ satisfies 
 \[
  \left\|\chi_{r_1,r_2} v\right\|_{Y(\Rm)} \lesssim_1 \left(\frac{r_2-r_1}{R}\right)^\frac{1}{10} \|u_1\|_{L^2(\Rm^3)}. 
 \]
\end{lemma}
\begin{proof}
 We recall the explicit formula 
 \[
  v(r,t) = \frac{1}{2r} \int_{r-t}^{r+t} s u_1(s) {\rm d} s, \qquad r>|t|. 
 \]
 It follows that 
 \[
  |v(r,t)| \lesssim_1 \frac{|t|^{1/2}}{r} \|su_1(s)\|_{L^2(\Rm^+)} \lesssim_1 r^{-1/2} \|u_1\|_{L^2(\Rm^3)}. 
 \]
 Again we always have $v(r,t)=0$ if $|t|<r<R/2$. Thus we also have 
 \[
  |v(r,t)| \lesssim_1 \min\left\{r^{-1/2}, R^{-1/2}\right\} \|u_1\|_{L^2(\Rm^3)}, \qquad r>|t|. 
 \]
 A similar calculation to the proof of Lemma \ref{lemma csg} then completes the proof. 
\end{proof}

\begin{corollary} \label{lemma split data}
 Assume that $0\leq r_1 < r_2 \leq R$. Let $F\in L^1 L^2(\Rm \times \Rm^3)$ be radial and supported in the region $\Omega_R$. Then the solution $v$ to the linear free wave  
 \[
  \left\{\begin{array}{ll} \square v = F, & (x,t) \in \Rm^3 \times \Rm; \\ (v,v_t)|_{t=0} = (0,0) & \end{array}\right.
 \]
 satisfies 
 \[
  \|\chi_{r_1,r_2} v\|_{Y(\Rm)} \lesssim_1 \left(\frac{r_2-r_1}{R}\right)^\frac{1}{10} \|F\|_{L^1 L^2(\Rm \times \Rm^3)}. 
 \]
\end{corollary}
\begin{proof}
 This is a direct consequence of Lemma \ref{lemma csit} and Duhamel's formula 
 \[
  v(\cdot,t) = \int_{0}^t \mathbf{S}_L(t-t') (0,F(\cdot,t')) {\rm d} t'. 
 \]
\end{proof}

Finally we give a few estimates for the norm of a free wave, as well as its interaction strength with ground states, in term of its radiation concentration. 
\begin{lemma} \label{lemma vL bound}
 Let $v_L$ be a radial linear free wave with radiation profile $G(s)$ and $\lambda\in \Rm^+$. We define 
 \[
  \tau = \left(\sup_{0<r<\lambda} \frac{\lambda}{r} \int_{-r}^r |G(s)|^2 {\rm d} s\right)^{1/2}  + \sup_{r>0} \frac{1}{r^{1/2}} \int_{-r}^r |G(s)| {\rm d} s. 
 \]
 Then 
 \begin{itemize}
  \item For $0\leq r_1 < r_2 \leq \lambda$, we have 
  \[
   \|\chi_{r_1,r_2} v_L\|_{Y(\Rm)} \lesssim_1 \left(\frac{r_2-r_1}{\lambda}\right)^{1/10} \tau. 
  \]
  \item If $1\leq r_1 < r_2 \leq \lambda$, then 
  \[
   \left\|\chi_{r_1,r_2} W^4 v_L\right\|_{L^1 L^2} \lesssim_1 \tau \lambda^{-1/2} r_1^{-3/2}. 
  \]
  \item If $0\leq r_1 < r_2 \leq 2 < \lambda$, then 
  \[
   \left\|\chi_{r_1,r_2} W^4 v_L\right\|_{L^1 L^2} \lesssim_1 \tau \lambda^{-1/2} (r_2-r_1)^{1/2}.
  \]
 \end{itemize}
\end{lemma}
\begin{proof}
 First of all, we give a point-wise estimate of $v_L(r,t)$ for $r>|t|$. If $r+|t|<\lambda$, then 
 \begin{align*}
  |v_L(r,t)| & = \frac{1}{r} \left| \int_{t-r}^{t+r} G(s) {\rm d} s\right|  \lesssim_1 \frac{1}{r^{1/2}} \|G\|_{L^2(t-r,t+r)} \\
  & \lesssim_1 \frac{1}{r^{1/2}} \|G\|_{L^2(-|t|-r,|t|+r)} \lesssim_1 \frac{1}{r^{1/2}} \cdot \left(\frac{|t|+r}{\lambda}\right)^{1/2} \tau \lesssim_1 \frac{\tau}{\lambda^{1/2}}. 
 \end{align*}
 On the other hand, we always have 
 \[
  |v_L(r,t)| \leq  \frac{1}{r} \int_{t-r}^{t+r} |G(s)| {\rm d} s \lesssim_1 \frac{1}{r^{1/2} (|t|+r)^{1/2}}  \int_{-|t|-r}^{|t|+r} |G(s)| {\rm d} s \lesssim_1 \frac{\tau}{r^{1/2}}.
 \] 
 In summary we have 
 \[
  |v_L(r,t)| \lesssim_1 \min\left\{\lambda^{-1/2}, r^{-1/2}\right\} \tau, \qquad r>|t|. 
 \]
 Thus we may conduct a direct calculation for $0\leq r_1 < r_2 \leq \lambda$
 \begin{align*}
  \|\chi_{r_1,r_2} v_L\|_{Y(\Rm)}^5 & \lesssim_1 \int_{\Rm} \left(\int_{r_1+|t|}^{r_2+|t|} \min\left\{\lambda^{-5}, r^{-5}\right\}\tau^{10} \cdot r^2 {\rm d}r \right)^{1/2} {\rm d} t\\
  & \lesssim_1 \int_{-\lambda}^\lambda \left(\int_{r_1+|t|}^{r_2+|t|} \lambda^{-5} \tau^{10} \cdot r^2 {\rm d}r \right)^{1/2} {\rm d} t + \int_{|t|>\lambda} \left(\int_{r_1+|t|}^{r_2+|t|} r^{-3} \tau^{10} {\rm d}r \right)^{1/2} {\rm d} t \\
  & \lesssim_1 \left(\frac{r_2-r_1}{\lambda}\right)^{1/2} \tau^5.
 \end{align*}
 Next we assume that $1\leq r_1 < r_2 \leq \lambda$. A direct calculation shows that 
 \begin{align*}
  \left\|\chi_{r_1,r_2} W^4 v_L\right\|_{L^1 L^2} & \lesssim_1 \int_{\Rm} \left(\int_{r_1+|t|}^{r_2+|t|} \min\left\{\lambda^{-1}, r^{-1}\right\}\tau^{2} r^{-8} \cdot r^2 {\rm d}r \right)^{1/2} {\rm d} t\\
  & \lesssim_1 \int_{\Rm} \left(\int_{r_1+|t|}^{r_2+|t|} \lambda^{-1} \tau^{2} r^{-6} {\rm d}r \right)^{1/2} {\rm d} t \\
  & \lesssim_1 \int_{\Rm} \lambda^{-1/2} \tau (r_1+|t|)^{-5/2} {\rm d} t\\
  & \lesssim_1 \lambda^{-1/2} \tau r_1^{-3/2}.
 \end{align*}
 Finally if $0\leq r_1 < r_2 \leq 2 < \lambda$, then we have 
  \begin{align*}
  \left\|\chi_{r_1,r_2} W^4 v_L\right\|_{L^1 L^2} & \lesssim_1 \int_{\Rm} \left(\int_{r_1+|t|}^{r_2+|t|} \min\left\{\lambda^{-1}, r^{-1}\right\}\tau^{2} W(r)^8 \cdot r^2 {\rm d}r \right)^{1/2} {\rm d} t\\
  & \lesssim_1 \int_{-1}^1 \left(\int_{r_1+|t|}^{r_2+|t|} \lambda^{-1} \tau^{2} {\rm d}r \right)^{1/2} {\rm d} t + \int_{|t|>1} \left(\int_{r_1+|t|}^{r_2+|t|} \lambda^{-1} \tau^{2} r^{-6} {\rm d}r \right)^{1/2} {\rm d} t \\
  & \lesssim_1 \int_{-1}^1 \lambda^{-1/2} \tau (r_2-r_1)^{1/2} {\rm d} t + \int_{|t|>1} \left(\frac{\lambda^{-1} \tau^2 (r_2-r_1)}{(r_1+|t|)^6}\right)^{1/2} {\rm d} t\\
  & \lesssim_1 \tau \lambda^{-1/2} (r_2-r_1)^{1/2}.
 \end{align*}
\end{proof} 

Finally we give a Strchartz estimate with highly localized radiation profile. 
\begin{lemma}[see Lemma 5.1 of Shen \cite{dynamics3d}] \label{calculation 1}
 Let $v_L$ be a radial free wave and $I = [a,b] \subset \Rm^+$ be an interval. Then 
 \[
  \|\chi_0 v_L\|_{Y(\Rm)} \lesssim_1  \|G_+\|_{L^2(\Rm\setminus I)} + \left(\frac{b-a}{a}\right)^{1/2} \|G_+\|_{L^2(I)}. 
 \]
 Here $G_+$ is the radiation profile of $v_L$ in the positive time direction. 
 \end{lemma}
 
\section{Soliton resolution estimates} 

The following proposition gives an instantaneous soliton resolution of a radial solution to (CP1), as well as some quantitative properties of the soliton resolution, as long as the solution is asymptotically equivalent to a free wave with a small Strichartz norm outside the main light cone. 

\begin{proposition} \label{main tool} 
 Let $n$ be a positive integer and $c_2 \gg 1$ be a sufficiently large constant. Then there exists a small constant $\delta_0 = \delta_0(n,c_2)>0$, such that if a radial exterior solution $u$ to (CP1) defined in $\Omega_0$ is asymptotically equivalent to a finite-energy free wave $v_L$ with $\delta \doteq \|\chi_0 v_L\|_{Y(\Rm)} < \delta_0$, then one of the following holds: 
 \begin{itemize}
  \item [(a)] There exists a sequence $(\zeta_j, \lambda_j)\in \{+1,-1\}\times \Rm^+$ for $j=1,2,\cdots,J$ with $0 \leq J\leq n-1$ such that  
  \begin{align*}
    \frac{\lambda_{j+1}}{\lambda_j}  \lesssim_{j,c_2} \delta^2, \qquad & j=1,2,\cdots, J-1; \\
   \left\|\vec{u}(\cdot, 0)-\sum_{j=1}^J \zeta_j (W_{\lambda_j},0) - \vec{v}_L(\cdot,0)\right\|_{\dot{H}^1\times L^2} & + \left\|\chi_0 \left(u - \sum_{j=1}^J \zeta_j W_{\lambda_j} \right)\right\|_{Y(\Rm)}  \lesssim_{J,c_2} \delta.
  \end{align*}
  \item[(b)] There exists a sequence $(\zeta_j, \lambda_j)\in \{+1,-1\}\times \Rm^+$ for $j=1,2,\cdots,n$ satisfying 
  \[
     \frac{\lambda_{j+1}}{\lambda_j} \lesssim_{j,c_2} \delta^2, \qquad j=1,2,\cdots, n-1;
  \]
  such that $u$ satisfies the following soliton resolution estimate in the exterior region  
  \begin{align*}
    \left\|\vec{u}(\cdot,0)-\sum_{j=1}^n \zeta_j (W_{\lambda_j},0) - \vec{v}_L(\cdot,0)\right\|_{\mathcal{H}(c_2 \lambda_n)} + \left\|\chi_{c_2 \lambda_n} \left(u - \sum_{j=1}^n \zeta_j W_{\lambda_j} \right)\right\|_{Y(\Rm)} & \lesssim_{n,c_2} \delta.
  \end{align*}
 \end{itemize}
 \end{proposition} 

\begin{remark} 
 Please refer to \cite{dynamics3d} for the proof of this proposition. Please note that the proposition here is slightly different from the original one (Proposition 4.1 in \cite{dynamics3d}) in three aspects 
 \begin{itemize}
  \item In Proposition 4.1 of \cite{dynamics3d} we fix an absolute constant $c_2$, thus the implicit constants in the inequalities there only depends on the integers $j, J$ or $n$. In the further argument of this work we probably have to choose a larger parameter $c_2$ than the original one used in \cite{dynamics3d}, thus we allow to choose any sufficiently large parameter $c_2$ here. A brief review of the proof given in \cite{dynamics3d} reveals that the proof applies to all large constants $c_2$. 
  \item Proposition 4.1 in \cite{dynamics3d} applies to all (weakly) asymptotically equivalent solutions of a free wave with a small Strichartz norm, even if those solutions are not necessarily defined in the whole region $\Omega_0$. For simplicity we only consider exterior solutions defined in the whole region $\Omega_0$ in this work. Please see Remark 4.2 of \cite{dynamics3d} for more details. 
  \item Instead of two parameters $\zeta$ and $\lambda$, the original proposition utilize a single parameter $\alpha$ to represent a ground state 
  \[
   W^\alpha = \frac{1}{\alpha} \left(\frac{1}{3} + \frac{|x|^2}{\alpha^4}\right)^{-1/2}, \qquad \alpha \in \Rm\setminus \{0\}. 
  \]
  These two ways of representation are completely equivalent. It is not difficult to see  
  \[
   W^\alpha = \zeta W_\lambda \quad \Longleftrightarrow \quad \alpha = \zeta \lambda^{1/2}.  
  \]
 \end{itemize}
\end{remark} 

\begin{remark} \label{uniform scale ratio}
 Please note that the scale parameters $\lambda_j$ not only depend on the exterior solution $u$, but also the choice of $c_2$. In fact, a brief review of the proof shows that the parameters $\lambda_j$ and $\zeta_j$ are determined inductively by the identities ($\lambda_0 = +\infty$)
 \begin{align*}
  \lambda_j & = c_2^{-1} \max\left\{r\in (0,\lambda_{j-1}): r^{1/2} \left|u(r,0) - v_L(r,0) - \sum_{k=1}^{j-1} \zeta_k W_{\lambda_k}(r)\right| = c_2^{1/2} W(c_2) \right\}; \\
  \zeta_j & = {\rm sign} \left(u(c_2 \lambda_j,0) - v_L(c_2 \lambda_j,0) - \sum_{k=1}^{j-1} \zeta_k W_{\lambda_k}(c_2 \lambda_j)\right); 
 \end{align*}
 which implies that 
 \[
  u(c_2 \lambda_j,0) - v_L(c_2 \lambda_j,0) - \sum_{k=1}^j W_{\lambda_k} (c_2 \lambda_j) = 0.  
 \]
 Nevertheless, the number of bubbles $J$, the signs $\zeta_j$ do not depend on the choice of the large parameter $c_2$, as long as $\delta < \delta(n,c_2)$ is sufficiently small. In addition, we may also choose the implicit constant in the ratio inequality 
 \[
  \frac{\lambda_{j+1}}{\lambda_j} \lesssim_j \delta^2, \qquad j=1,2,\cdots,J-1
 \]
 independent of $c_2$, under the same assumption. In fact, we may fix a large constant $c_2^\ast$ and consider another constant $c_2 > c_2^\ast$. According to Proposition \ref{main tool}, if $\delta < \delta(n, c_2, \varepsilon)$ is sufficiently small, where $\varepsilon$ is a parameter to be determined later, we may apply Proposition \ref{main tool} with each parameter $c_2, c_2^\ast$ and obtain
 \begin{align*}
   \left\|\vec{u}(\cdot, 0)-\sum_{j=1}^J \zeta_j (W_{\lambda_j},0) - \vec{v}_L(\cdot,0)\right\|_{\dot{H}^1\times L^2}  & \leq \varepsilon, & & \hbox{(Case a)} \\
   \left\|\vec{u}(\cdot, 0)-\sum_{j=1}^n \zeta_j (W_{\lambda_j},0) - \vec{v}_L(\cdot,0)\right\|_{\mathcal{H}(c_2 \lambda_n)} & \leq \varepsilon; & & \hbox{(Case b)}
 \end{align*}
 as well as
  \begin{align*}
   \left\|\vec{u}(\cdot, 0)-\sum_{j=1}^{J^\ast} \zeta_j^\ast (W_{\lambda_j^\ast},0) - \vec{v}_L(\cdot,0)\right\|_{\dot{H}^1\times L^2}  & \leq \varepsilon, & & \hbox{(Case a)} \\
   \left\|\vec{u}(\cdot, 0)-\sum_{j=1}^n \zeta_j^\ast (W_{\lambda_j^\ast},0) - \vec{v}_L(\cdot,0)\right\|_{\mathcal{H}(c_2^\ast \lambda_n^\ast)} & \leq \varepsilon; & & \hbox{(Case b)}
 \end{align*}
 with 
  \begin{align*}
  &\frac{\lambda_{j+1}}{\lambda_{j}} \leq \varepsilon^2; & &\frac{\lambda_{j+1}^\ast}{\lambda_{j}^\ast} \leq \varepsilon^2. 
 \end{align*}
 Here $(\zeta_j,\lambda_j)$ and $(\zeta_j^\ast, \lambda_j^\ast)$ are the signs and scales with parameters $c_2, c_2^\ast$, respectively.  A combination of the two soliton resolution representations yields that 
 \[
  \left\|\sum_{j=1}^J \zeta_j (W_{\lambda_j},0) - \sum_{j=1}^{J^\ast} \zeta_j^\ast (W_{\lambda_j^\ast},0)\right\|_{\mathcal{H}(R)} \leq 2 \varepsilon. 
 \]
 Here we let $J = n$ and/or $J^\ast = n$ if the corresponding soliton resolution is in case b. The radius $R$ is chosen to be 
 \begin{align*}
  R = \left\{\begin{array}{ll} +\infty, & J, J^\ast < n; \\ c_2 \lambda_n, & J=n, J^\ast < n; \\ c_2^\ast \lambda_n^\ast, & J<n, J^\ast =n; \\
  \max\{c_2 \lambda_n, c_2^\ast \lambda_n^\ast\}, & J=J^\ast=n. \end{array}\right.
 \end{align*}
  It is not difficult to see that if $\varepsilon = \varepsilon(n, c_2)$ is a sufficiently small constant, then we must have  
  \begin{align*}
   &J = J^\ast;& &\zeta_j = \zeta_j^\ast;& &\lambda_j \simeq_1 \lambda_j^\ast. 
  \end{align*}
  It immediately follows that 
  \[
   \frac{\lambda_{j+1}}{\lambda_j} \lesssim_1 \frac{\lambda_{j+1}^\ast}{\lambda_j^\ast} \lesssim_j \delta^2. 
  \]
 \end{remark}

\paragraph{J-bubble exterior solutions} Fix a constant $c_2$ as above. We call a radial exterior solution $u$ to (CP1) defined in the region $\Omega_0$ to be a $J$-bubble exterior solution if it satisfies the following conditions. 
\begin{itemize}
 \item $u$ is asymptotically equivalent to a free wave $v_L$ with $\delta \doteq \|\chi_0 v_L\|_{Y(\Rm)} < \delta_0(J+1,c_2)$. Here $\delta_0(J+1,c_2)$ is the constant given in Proposition \ref{main tool}. 
 \item The soliton resolution of $u$ given in Proposition \ref{main tool} comes with exactly $J$ bubbles. 
\end{itemize}
Please note that this definition only considers the instantaneous soliton resolution property given in Proposition \ref{main tool}, thus is different from the conception of $J$-bubble soliton resolution as time tends to a blow-up time or infinity. However, if $u$ is a radial global solution whose soliton resolution comes with $J$ bubbles as $t\rightarrow +\infty$, then the time-translated solution $u(\cdot, \cdot + t)$ must be a $J$-bubble exterior solution defined here when $t$ is sufficiently large, as shown in the last section of this work. The situation of type II blow-up solution is similar, if we apply a local cut-off technique when necessary. More details are given in the last section. 

\begin{remark}\label{square estimate}
 Let $u$ be a $J$-bubble exterior solution as defined above. Then we may define 
 \[
  w_\ast (x,t) = u(x,t) - \sum_{j=1}^J \zeta_j W_{\lambda_j} (x) - v_L(x,t),
 \]
 which means 
 \[
  \square w_\ast = F(u) -\sum_{j=1}^J \zeta_j F(W_{\lambda_j}),
 \]
 and obtain the following estimate if $\delta < \delta(J,c_2)$ is sufficiently small:
 \begin{align*}
  & \left\|\chi_0 \left(F(u) -\sum_{j=1}^J \zeta_j F(W_{\lambda_j})\right)\right\|_{L^1 L^2} \\
  & \leq \left\|\chi_0 \left(F(u) -F\left(\sum_{j=1}^J \zeta_j W_{\lambda_j}\right)\right)\right\|_{L^1 L^2} + \left\|\chi_0 \left(F\left(\sum_{j=1}^J \zeta_j W_{\lambda_j}\right) - \sum_{j=1}^J \zeta_j F(W_{\lambda_j})\right)\right\|_{L^1 L^2} \\
  & \lesssim_J \left(\|\chi_0 w_\ast\|_{Y(\Rm)}^4 + \|\chi_0 W\|_{Y(\Rm)}^4 + \|\chi_0 v_L\|_{Y(\Rm)}^4 \right) \left(\|\chi_0 w_\ast\|_{Y(\Rm)}+\|\chi_0 v_L\|_{Y(\Rm)}\right) \\
  & \quad  + \sum_{1\leq j < m\leq J} \left(\left\|\chi_0 W_{\lambda_j}^4 W_{\lambda_m}\right\|_{L^1 L^2} + \left\|\chi_0 W_{\lambda_j} W_{\lambda_m}^4 \right\|_{L^1 L^2}\right)\\
  & \lesssim_{J,c_2} \delta + + \sum_{1\leq j < m\leq J} \left(\frac{\lambda_m}{\lambda_j}\right)^{1/2} \\
  & \lesssim_{J,c_2} \delta. 
 \end{align*}
 Here we apply Corollary \ref{global interaction bubbles} and the dilation invariance. 
\end{remark}

\section{Linearized equation}

In this section we consider the (approximated) linearized equation of the error function 
\[
 w_\ast = u - v_L - \sum_{j=1}^J \zeta_j W_{\lambda_j}.
\]
Applying the operator $\square = \partial_t^2 - \Delta$ on both sides, we obtain 
\begin{align*}
 \square w_\ast & = \square u - \sum_{j=1}^J \zeta_j F(W_{\lambda_j}) \\
 & = F\left(w_\ast+ v_L + \sum_{j=1}^J \zeta_j W_{\lambda_j}\right)  - \sum_{j=1}^J \zeta_j F(W_{\lambda_j})\\ 
 & \approx 5 W_{\lambda_J}^4 w_\ast + 5 \zeta_{J-1} W_{\lambda_J}^4 W_{\lambda_{J-1}}\\
 & \approx 5 W_{\lambda_J}^4 w_\ast + 5 \sqrt{3} \zeta_{J-1} \lambda_{J-1}^{-1/2} W_{\lambda_J}^4. 
\end{align*}
Here we discard all insignificant terms and use the following approximation in the energy channel $|x|-|t| \simeq \lambda_{J}$
\[
 W_{\lambda_{J-1}} \approx \lambda_{J-1}^{-1/2}  \sqrt{3}.  
\]
 Since the ground states do not depend on the time and $w_\ast$ is not expected to send any radiation, we expect that the major part of $w_\ast$ is also independent of time. Therefore we are interested to the solution to the linear elliptic equation
\[ 
 -\Delta \varphi = 5 W^4 \varphi + 5 W^4
\]
Here we assume $\lambda_J = 1$ without loss of generality, by the natural dilation, and temporarily ignore the insignificant constant $\sqrt{3} \zeta_{J-1} \lambda_{J-1}^{-1/2}$. 

\begin{lemma} \label{linearized elliptic}
 There exist three constants $c_5 \gg 1$, $r_4 \ll 1$ and $\mu_0 \in \Rm \setminus \{0\}$, such that for any parameter $c > c_5$, the elliptic equation 
 \[
  -\Delta \varphi = 5 W^4 \varphi + 5 W^4
 \]
 admits a solution $\varphi \in \mathcal{C}^2 (\Rm^3 \setminus \{0\})$ satisfying the following conditions 
 \begin{itemize}
  \item $\varphi$ is a radially symmetric solution;
  \item $\varphi(c) = 0$; 
  \item $1/2 \leq r\varphi (r) /\mu_0 \leq 3/2$ for all $r\in (0,r_4)$; 
  \item $|\varphi(x)|\lesssim_1 |x|^{-1}$ for all $x\in \Rm^3 \setminus\{0\}$; here the implicit constant is independent of $c>c_5$; 
  \item $\varphi \in \dot{H}^1(\{x: |x|>r\})$ for any $r>0$.
 \end{itemize}
\end{lemma}
\begin{proof}
 It follows from a direct calculation that $u$ satisfies the elliptic equation above if and only if $w(r) = r \varphi(r)$ satisfies the one-dimensional elliptic equation 
 \begin{equation} \label{one dimension elliptic} 
  - w_{rr} = 5 \left(\frac{1}{3} + r^2\right)^{-2} (w+r).
 \end{equation}
 We first construct a special solution $w^\ast$ to this equation with the best decay near the infinity. We consider the map $\mathbf{T} : C([1,+\infty)) \rightarrow C([1,+\infty))$:
 \[
  (\mathbf{T} w) (r)  = - 5 \int_r^\infty \int_s^\infty  \left(\frac{1}{3} + \tau^2\right)^{-2} (w(\tau)+\tau) {\rm d} \tau {\rm d} s. 
 \]
 A direct calculation shows that the norms in the Banach space $X = C([1,+\infty))$ satisfy 
 \begin{align*}
  \|\mathbf{T} w\|_{X} \leq \sup_{r\geq 1} \left(5  \int_r^\infty \int_s^\infty  \tau^{-4} (\|w\|_X +\tau) {\rm d} \tau {\rm d} s\right) \leq \frac{5}{2} + \frac{5}{6}\|w\|_X 
 \end{align*}
 and
 \begin{align*}
  \|\mathbf{T} w_1 - \mathbf{T} w_2 \|_{X} &\leq \sup_{r\geq 1} \left(5 \int_r^\infty \int_s^\infty  \tau^{-4} \|w_1 - w_2\|_X {\rm d} \tau {\rm d} s\right) \\
  & \leq \frac{5}{6} \|w_1 - w_2\|_X. 
 \end{align*}
 This verifies that $\mathbf{T}$ is a contraction map. The classic fixed-point argument immediately gives a solution $w^\ast$ to \eqref{one dimension elliptic} with the asymptotic behaviour 
 \begin{align*}
  & w^\ast(r) = - \frac{5}{2} r^{-1} + O(r^{-2}); & & w_r^\ast(r) = \frac{5}{2} r^{-2} + O(r^{-3}). 
 \end{align*}
 Since \eqref{one dimension elliptic} is a linear ordinary differential equation with bounded coefficients, the solution $w^\ast$ may extend to a solution defined in $[0,+\infty)$. Now we claim that $w^\ast(0) \neq 0$. Indeed, we observe that 
 \begin{itemize} 
  \item one of the solution to \eqref{one dimension elliptic} is $-r$, which is zero at the origin;
  \item $v = r\left(r^2 - \frac{1}{3}\right)\left(\frac{1}{3} + r^2\right)^{-3/2}$ satisfies the homogeneous differential equation 
  \[
   - v_{rr} = 5 \left(\frac{1}{3} + r^2\right)^{-2} v. 
  \]
 \end{itemize}
 As a result, all solutions to \eqref{one dimension elliptic} with $w(0) = 0$ can be given by 
 \[
  w(r) = -r +C v(r)\qquad \Longrightarrow \qquad \lim_{r\rightarrow +\infty} w(r) = -\infty. 
 \]
 The solution $w^\ast(r)$ is clearly not in this form, thus we must have $\mu_0 \doteq w^\ast (0) \neq 0$. This verifies our claim. Finally we may give the desired solution $w$ in the following form and let $\varphi(x) = |x|^{-1} w(|x|)$.
 \begin{align*}
  &w(r) = w^\ast (r) + \beta v(r), & & \beta = - \frac{w^\ast(c)}{v(c)}. 
 \end{align*}
 This clearly solves the equation \eqref{one dimension elliptic} with $w(c) = 0$ and $w\in \mathcal{C}^2(\Rm^+)$. By the asymptotic behaviour of $w^\ast$ and $v$, we may choose a sufficiently large absolute constant $c_5\gg 1$, such that  
 \[
  |\beta| \simeq_1 c^{-1}, \qquad c>c_5.
 \]
 This immediately implies that $|w(r)| \leq |w^\ast(r)| + |\beta| |v(r)|$ is uniformly bounded for all $c > c_5$ and $r>0$; and that 
 \[
  |w(r) - \mu_0| \leq  |w^\ast(r) - \mu_0| + |\beta| |v(r)| \leq  |w^\ast(r) - \mu_0| + \sqrt{3} |\beta| r <  |\mu_0|/2, \qquad \forall r \in (0,r_4),
 \]
 as long as the constant $r_4$ is sufficiently small. The first four properties of $\varphi(x) = |x|^{-1} w(|x|)$ immediately follows these properties of $w$. Finally the last property of $\varphi$ follows from the regularity $\varphi \in \mathcal{C}^2(\Rm\setminus \{0\})$ and the asymptotic behaviour
 \[
  |\varphi_r (r)| \leq \frac{1}{r}|w_r^\ast(r)| + \frac{1}{r^2}|w^\ast(r)| + |\beta|\left|\partial_r (r^{-1} v)\right| \lesssim_1 r^{-2}, \qquad r\gg 1. 
 \] 
\end{proof}

\begin{remark} \label{remark delta varphi}
 Let $\varphi$ be the solution given in the previous lemma. The equation $-\Delta \varphi = 5 W^4 \varphi + 5 W^4$ and the uniform upper bound $|\varphi(x)| \lesssim_1 |x|^{-1}$ implies that 
 \[
  \|\chi_0 \Delta \varphi\|_{L^1 L^2(\Rm \times \Rm^3)} \lesssim_1 1. 
 \]
 The upper bound is independent of large parameter $c > c_5$. 
\end{remark}

\begin{corollary}\label{w lambda varphi}
 Let $0\leq r_1<r_2\leq \lambda$. Then the solution $\varphi$ given in Lemma \ref{linearized elliptic} satisfies 
 \[
  \left\|\chi_{r_1,r_2} W_\lambda^4 \varphi\right\|_{L^1 L^2} \lesssim_1  \lambda^{-1} (r_2-r_1)^{1/2}.
 \]
 Please note that the implicit constant is independent of $c > c_3$. 
\end{corollary}
\begin{proof}
 Let us recall that $|\varphi(x)| \lesssim_1 |x|^{-1}$. Thus a direct calculation shows that 
 \begin{align*}
   \left\|\chi_{r_1,r_2} W_\lambda^4 \varphi\right\|_{L^1 L^2} & \lesssim_1 \int_{-\infty}^\infty \left(\int_{r_1+|t|}^{r_2+|t|} \frac{1}{\lambda^4} \left(\frac{1}{3} + \frac{r^2}{\lambda^2}\right)^{-4} r^{-2}\cdot r^2 {\rm d} r\right)^{1/2} {\rm d} t\\
   & \lesssim_1 \int_{-\lambda}^\lambda \left(\int_{r_1+|t|}^{r_2+|t|} \frac{1}{\lambda^4} {\rm d} r\right)^{1/2} {\rm d} t + \int_{|t|>\lambda} \left(\int_{r_1+|t|}^{r_2+|t|} \frac{\lambda^4}{r^8} {\rm d} r\right)^{1/2} {\rm d} t\\
   & \lesssim_1 \int_{-\lambda}^\lambda  \frac{(r_2-r_1)^{1/2}}{\lambda^2} {\rm d} t + \int_{|t|>\lambda} \frac{\lambda^2 (r_2-r_1)^{1/2}}{(r_1+|t|)^4} {\rm d} t\\
    & \lesssim_1 \lambda^{-1} (r_2-r_1)^{1/2}.
 \end{align*}
\end{proof}

\section{Bubble Interaction}

In this section we give a few approximation of a $J$-bubble exterior solution $u$ with weak radiation concentration, if such a solution existed. 

\begin{lemma} \label{key connection}
 Fix a positive integer $J\geq 2$. There exists an absolute constant $c_2$ and two small constants $c_1=c_1(J)$ and $\tau_0 = \tau_0 (J)>0$, such that if a radial exterior solution $u$ to (CP1) is a $J$-bubble exterior solution defined in Section 3 (with the parameter $c_2$) satisfying
 \begin{align*}
  \tau \doteq & \left(\sup_{0<r<\lambda_{J-1}} \!\!\! \frac{\lambda_{J-1}}{r} \int_{-r}^r |G(s)|^2 {\rm d} s\right)^{1/2}  + \|\chi_0 v_L\|_{Y(\Rm)} + \left\|\vec{u}(0) - \vec{v}_L(0) - \sum_{j=1}^J \zeta_j (W_{\lambda_j}, 0)\right\|_{\mathcal{H}(c_2 \lambda_{J})} \\
  & + \sup_{r>0} \frac{1}{r^{1/2}} \int_{-r}^r |G(s)| {\rm d} s + \left\|\chi_{c_2 \lambda_{J}}\left(F(u) - \sum_{j=1}^{J} \zeta_j F(W_{\lambda_j})\right)\right\|_{L^1 L^2(\Rm\times \Rm^3)} < \tau_0,
 \end{align*}
 where $G$ is the radiation profile of the asymptotic equivalent free wave $v_L$ of $u$; $\zeta_j$'s and $\lambda_j$'s are the corresponding signs and scales given by Proposition \ref{main tool}; then the error term 
 \[
  w = u - \sum_{j=1}^J \zeta_j W_{\lambda_j}(x) - \sqrt{3} \zeta_{J-1} \lambda_{J-1}^{-1/2} \varphi (x/\lambda_J) - v_L(x,t)
 \]
 and the radiation profile $G^\ast$ associated to $\vec{w}(0)$ satisfy the inequalities
 \begin{align*}
  \|G^\ast\|_{L^2(s: 2^k c_2 \lambda_J < |s| < 2^{k+1} c_2 \lambda_J)} + \|\chi_{2^k c_2 \lambda_J, 2^{k+1} c_2 \lambda_J} \square w\|_{L^1 L^2} & \leq 2^\frac{k-K}{2}\tau, \quad k=0,1,2,\cdots,K; \\
  \|\vec{w}(0)\|_{\mathcal{H}(c_1 \lambda_{J-1})} + \left\|\chi_{c_1 \lambda_{J-1}} \square w\right\|_{L^1 L^2} & \leq 2 \tau. 
 \end{align*}
 Here $\varphi(x)$ is the solution to the elliptic equation $-\Delta \varphi = 5 W^4 \varphi + 5W^4$ given in Lemma \ref{linearized elliptic} with the parameter $c= c_2$; $K$ is the minimal positive integer such that $2^{K+1} c_2 \lambda_J \geq c_1 \lambda_{J-1}$. 
\end{lemma}
\begin{proof}
 By dilation we may assume $\lambda_J = 1$, without loss of generality. We let $c_2 \gg 1$ and $c_1 \ll 1$ be constants, which will be determined later in the argument. According to Remark \ref{uniform scale ratio}, the inequalities $\lambda_{J-1}^{-1/2} \lesssim_J \tau$ and $4c_2  < c_1 \lambda_{J-1}$ always hold as long as $\tau < \tau(J,c_2,c_1)$ is sufficiently small. Please note that $\tau(J,c_2,c_1)$ here represents a constant determined by $J$, $c_2$ and $c_1$ only, which might be different at different places. This kind of notations will be frequently used in the subsequent. Now we compare $u$ with 
 \[
  S^\ast(x,t) = \sum_{j=1}^J \zeta_j W_{\lambda_j}(x) + \zeta_{J-1} \sqrt{3} \lambda_{J-1}^{-1/2} \varphi(x) + v_L(x,t). 
 \]
It is clear that 
\[
 w = u-S^\ast = \left(u - v_L - \sum_{j=1}^J \zeta_j W_{\lambda_j}(x)\right) - \sqrt{3} \zeta_{J-1} \lambda_{J-1}^{-1/2} \varphi(x)
\] 
It immediately follows from our assumption and the inequality $\lambda_{J-1}^{-1/2} \lesssim_J \tau$ that  
\begin{align*}
  \|\vec{w}(0)\|_{\mathcal{H}(c_1 \lambda_{J-1})} + \|\chi_{c_1 \lambda_{J-1}}(\square w)\|_{L^1 L^2} & \leq \tau + \frac{\sqrt{3}}{\lambda_{J-1}^{1/2}} \left(\|\varphi\|_{\dot{H}^1(\{x: |x|>c_1 \lambda_{J-1}\})} + \|\chi_{c_1 \lambda_{J-1}} \Delta \varphi\|_{L^1 L^2}\right)\\
  & \leq \tau + c(J) \tau \left(\|\varphi\|_{\dot{H}^1(\{x: |x|>c_1 \lambda_{J-1}\})} + \|\chi_{c_1 \lambda_{J-1}} \Delta \varphi\|_{L^1 L^2}\right).
\end{align*}
Since we have 
\[
 \lim_{r\rightarrow +\infty} \left(\|\varphi\|_{\dot{H}^1(\{x: |x|> r\})} + \|\chi_{r} \Delta \varphi\|_{L^1 L^2}\right) = 0, 
\]
the following inequality holds as long as $\tau < \tau(J,c_2,c_1)$ is sufficiently small 
\begin{equation} \label{initial inequality w}
 \|\vec{w}(0)\|_{\mathcal{H}(c_1 \lambda_{J-1})} + \|\chi_{c_1 \lambda_{J-1}}(\square w)\|_{L^1 L^2(\Rm \times \Rm^3)} \leq 2 \tau. 
\end{equation}
A similar argument, as well as the uniform boundedness of $\|\chi_0 \Delta \varphi\|$ given in Remark \ref{remark delta varphi}, also gives 
\begin{align} \label{upper bound of b1}
 \|\chi_{c_2}(\square w)\|_{L^1 L^2(\Rm \times \Rm^3)} \leq \tau + \frac{\sqrt{3}}{\lambda_{J-1}^{1/2}} \|\chi_{c_2} \Delta \varphi\|_{L^1 L^2} \lesssim_J \tau. 
\end{align}
For convenience we utilize the notations $\Psi_k$ for the following channel-like regions
\begin{align*}
 \Psi_k = \left\{(x,t): |t|+2^k c_2 < |x| < |t|+2^{k+1} c_2\right\}, \qquad k=0,1,2,\cdots, K.
\end{align*}
In order to take the advantage of finite speed of wave propagation, we also define
\begin{align*}
 \Psi_{k,\ell} & = \{(x,t)\in \Psi_k:  |x|+|t|<2^{k+\ell} c_2\}, & & \ell =1,2,\cdots, K+1-k.
\end{align*}
Next we introduce a few notations for the norms:
\begin{align*}
 & a_k = \|\chi_{\Psi_k} w\|_{Y(\Rm)}; & & a_{k,\ell} = \|\chi_{\Psi_{k,\ell}} w\|_{Y(\Rm)}; 
\end{align*}
as well as
\begin{align*}
 b_k = \|G^\ast\|_{L^2(\{s: 2^k c_2 < |s|<2^{k+1} c_2\})} + \|\chi_{\Psi_k} (\square w)\|_{L^1 L^2}. 
\end{align*}
Now we prove a few inequalities concerning $a_k$, $a_{k,\ell}$ and $b_k$, which will finally lead to the conclusion of Lemma \ref{key connection}. First of all, we observe that $w$ is $c_2$-weakly non-radiative. Thus we may apply Lemma \ref{scatter profile of nonlinear solution} on $w$, recall \eqref{upper bound of b1} and obtain 
\begin{align} \label{upper bound of b2}
 b_{k} & \lesssim_1 \|\chi_{\Psi_k} (\square w)\|_{L^1 L^2} \lesssim_J \tau.
\end{align}
A more delicate upper bound of $b_k$ can be given in terms of $a_k$ and $a_{k,\ell}$. We calculate (we use the notation $\lambda = \lambda_{J-1}$ for convenience below)
\begin{align*}
 \square w & = F(u) -  \sum_{j=1}^J \zeta_j F(W_{\lambda_j}) - \zeta_{J-1} \frac{\sqrt{3}}{\lambda^{1/2}}(5W^4 \varphi +  5W^4) \\
 & = F\left(w + v_L + \sum_{j=1}^J \zeta_j W_{\lambda_j} + \frac{\sqrt{3} \zeta_{J-1}}{\lambda^{1/2}}  \varphi\right) - \sum_{j=1}^J \zeta_j F(W_{\lambda_j}) - \zeta_{J-1} \frac{5\sqrt{3}}{\lambda^{1/2}}(W^4 \varphi +  W^4).
\end{align*}
It immediately follows that 
\[
 b_k \lesssim_1 I_1 + I_2 + \cdots + I_7. 
\]
with
\begin{align*}
 I_1 & = \|\chi_{\Psi_k} W^4 w\|_{L^1 L^2};\\
 I_2 & = c(J) \left\|\chi_{\Psi_k}\left(|w|^5 + \lambda^{-2} \varphi^4 |w| + |v_L|^4 |w| + \sum_{j=1}^{J-1} W_{\lambda_j}^4 |w|\right)\right\|_{L^1 L^2};\\
 I_3 & = c(J) \left\|\chi_{\Psi_k}\left(|v_L|^5 + \lambda^{-2} \varphi^4 |v_L|  + \sum_{j=1}^{J} W_{\lambda_j}^4 |v_L|\right)\right\|_{L^1 L^2};\\
 I_4 & = c(J) \left\|\chi_{\Psi_k}\left(\lambda^{-5/2} |\varphi|^5 + \lambda^{-1} W^3 |\varphi|^2 + \lambda^{-1/2} \sum_{j=1}^{J-1} W_{\lambda_{j}}^4 |\varphi|\right)\right\|_{L^1 L^2};\\
 I_5 & = \left\|\chi_{\Psi_k}W^4(W_\lambda - \sqrt{3}\lambda^{-1/2})\right\|_{L^1 L^2};\\
 I_6 & = c(J) \left\|\chi_{\Psi_k}\left(\sum_{j=1}^{J-2} W^4 W_{\lambda_j} + \sum_{j=1}^{J-1} W W_{\lambda_j}^4 + W^3 W_{\lambda_{J-1}}^2 \right)\right\|_{L^1 L^2};\\
 I_7 & = c(J) \sum_{1\leq j< m\leq J-1} \left\|\chi_{\Psi_k} (W_{\lambda_j}^4 W_{\lambda_m} + W_{\lambda_j} W_{\lambda_m}^4)\right\|_{L^1 L^2}.
\end{align*}
We give the upper bounds of these terms one by one. First of all, we break the region $\Psi_k$ into several parts, apply Lemma \ref{upper bound of W norm} and deduce 
\begin{align*}
 I_1 & \lesssim_1 \left\|\chi_{\Psi_{k,1}} W^4 w\right\|_{L^1 L^2} + \sum_{\ell = 1}^{K-k} \left\|\chi_{\Psi_{k,\ell+1}\setminus \Psi_{k,\ell}} W^4 w\right\|_{L^1 L^2} + \left\|\chi_{\Psi_k \setminus \Psi_{k,K+1-k}} W^4 w\right\|_{L^1 L^2} \\
 & \lesssim_1 a_{k,1} \|\chi_{\Psi_k}W\|_{Y(\Rm)}^4 + \sum_{\ell = 1}^{K-k} \|\chi_{\Psi_k \setminus \Psi_{k,\ell}} W\|_{Y(\Rm)}^4 a_{k,\ell+1} + \|\chi_{\Psi_k \setminus \Psi_{k,K+1-k}} W\|_{Y(\Rm)}^4  a_k\\
 & \lesssim_1 a_{k,1} \left(c_2 2^k\right)^{-2} + \sum_{\ell =1}^{K-k} a_{k,\ell+1} \left(c_2 2^k\right)^{-2} 2^{-\frac{12}{5}\ell} + a_k \left(c_2 2^k\right)^{-2} 2^{-\frac{12}{5}(K-k)}\\
 & \lesssim_1 c_2^{-2} 2^{-2k} \left(\sum_{\ell =1}^{K-k+1}  2^{-\frac{12}{5}\ell} a_{k,\ell} + 2^{-\frac{12}{5}(K-k)} a_k\right). 
\end{align*}
In order to evaluate $I_2, I_3$, we recall Lemma \ref{lemma vL bound} and deduce 
\[
 \|\chi_{\Psi_k} v_L\|_{Y(\Rm)} \lesssim_1 \left(\frac{2^k c_2}{\lambda}\right)^{1/10} \tau \lesssim_1 c_1^{1/10} \tau 2^{\frac{k-K}{10}}.  
\]
Here we use the following fact, which will be frequently used in the subsequent argument 
\[
 2^K c_2 \simeq_1 c_1 \lambda \quad \Rightarrow \quad \frac{2^k c_2}{\lambda} = \frac{2^k 2^K c_2}{2^K \lambda} \simeq_1 \frac{2^k c_1\lambda}{2^K \lambda} \simeq_1 c_1 2^{k-K}. 
\]
In addition, we apply Lemma \ref{upper bound of W norm} and utilize the dilation invariance to deduce (as long as $\tau < \tau(J,c_2,c_1)$ is sufficiently small)
\begin{align*}
 \|\chi_{\Psi_k} W_{\lambda_j}\|_{Y(\Rm)} \lesssim_1 \left(\frac{c_2 2^k}{\lambda_j}\right)^{1/10} \left\{\begin{array}{ll} \lesssim_1 c_1^{1/10} 2^{\frac{k-K}{10}}, & j=J-1; \\ \lesssim_J c_1^{1/10} \tau^{1/5} 2^{\frac{k-K}{10}}, & j=1,2,\cdots,J-2. \end{array}\right.
\end{align*}
As a result, we obtain
\begin{align*}
 I_2 & \lesssim_J \|\chi_{\Psi_k} w\|_{Y(\Rm)}^5 + \lambda^{-2} \|\chi_{\Psi_k} \varphi\|_{Y(\Rm)}^4 \|\chi_{\Psi_k} w\|_{Y(\Rm)} + \|\chi_{\Psi_k} v_L\|_{Y(\Rm)}^4 \|\chi_{\Psi_k} w\|_{Y(\Rm)}\\
 & \qquad  + \sum_{j=1}^{J-1} \|\chi_{\Psi_k} W_{\lambda_j}\|_{Y(\Rm)}^4 \|\chi_{\Psi_k} w\|_{Y(\Rm)}\\
 & \lesssim_J a_k^5 + \lambda^{-2} a_k + c_1^{2/5} \tau^4 2^{\frac{2}{5}(k-K)} a_k + c_1^{2/5} 2^{\frac{2}{5}(k-K)} a_k\\
 & \lesssim_J a_k^5 + c_1^{2/5} 2^{\frac{2}{5}(k-K)} a_k. 
 \end{align*}
 Here we use the estimates 
 \begin{align*}
  &\|\chi_{\Psi_k} \varphi\|_{Y(\Rm)} \lesssim_1 \left\|\chi_{1} |x|^{-1}\right\|_{Y(\Rm)} \lesssim_1 1;& & \lambda \simeq_1 c_1^{-1} c_2 2^K. 
 \end{align*}
 Similarly we may combine the estimates above and Lemma \ref{lemma vL bound} to deduce 
 \begin{align*}
  I_3 & \lesssim_J \|\chi_{\Psi_k} v_L\|_{Y(\Rm)}^5 + \lambda^{-2} \|\chi_{\Psi_k} \varphi\|_{Y(\Rm)}^4 \|\chi_{\Psi_k} v_L\|_{Y(\Rm)} + \sum_{j=1}^{J-1} \|\chi_{\Psi_k} W_{\lambda_j}\|_{Y(\Rm)}^4 \|\chi_{\Psi_k} v_L\|_{Y(\Rm)} \\
  & \qquad + \left\|\chi_{\Psi_k} W^4 v_L\right\|_{L^1 L^2}\\
  & \lesssim_J c_1^{1/2} \tau^5 2^{\frac{1}{2}(k-K)} + c_1^{1/10} \lambda^{-2} \tau 2^{\frac{k-K}{10}} + c_1^{1/2} \tau 2^{\frac{1}{2}(k-K)} + \tau \lambda^{-1/2} (c_2 2^k)^{-3/2}\\
  & \lesssim_J c_1^{1/2} \tau 2^{\frac{1}{2}(k-K)}. 
 \end{align*}
 Next we apply Corollary \ref{w lambda varphi} and obtain 
  \begin{align*}
  I_4 & \lesssim_J \lambda^{-5/2} \|\chi_{\Psi_k} \varphi\|_{Y(\Rm)}^5 + \lambda^{-1} \|\chi_{\Psi_k} W\|_{Y(\Rm)}^3  \|\chi_{\Psi_k} \varphi\|_{Y(\Rm)}^2 + \lambda^{-1/2} \sum_{j=1}^{J-1} \left\|\chi_{\Psi_k} W_{\lambda_j}^4 \varphi\right\|_{L^1 L^2}\\
  & \lesssim_J \lambda^{-5/2} + \left(c_2 2^k\right)^{-3/2}  \lambda^{-1} + \lambda^{-1/2} \sum_{j=1}^{J-1} \lambda_j^{-1} (2^k c_2)^{1/2}\\
  & \lesssim_J \lambda^{-1} \lesssim_J c_1^{1/2} \tau 2^{\frac{k-K}{2}}.
 \end{align*}
 The estimates of $I_5$ and $I_6$ immediately follow from Lemma \ref{estimate W J lambda} and Lemma \ref{L1L2 channel estimate mixed}: 
 \begin{align*}
  I_5 & \lesssim_1 (2^k c_2)^{1/2} \lambda^{-5/2} \ln \lambda \lesssim_J c_1^{1/2} \tau 2^{\frac{k-K}{2}};\\
  I_6 & \lesssim_J \sum_{j=1}^{J-2} \lambda_j^{-1/2} (2^k c_2)^{-3/2} + \sum_{j=1}^{J-1} \lambda_j^{-1} (2^k c_2)^{1/2} + \lambda^{-1} (2^k c_2)^{-1/2} \lesssim_J c_1^{1/2} \tau 2^{\frac{k-K}{2}}. 
 \end{align*}
 Finally we combine Lemma \ref{L1L2 channel estimate mixed} and the dilation invariance to deduce 
 \begin{align*}
  I_7 & \lesssim_J \sum_{1\leq j < m\leq J-1} \left[\left(\frac{\lambda_j}{\lambda_m}\right)^{-1} \left(\frac{2^k c_2}{\lambda_m}\right)^{1/2} + \left(\frac{\lambda_j}{\lambda_m}\right)^{-1/2}\left(\frac{2^k c_2}{\lambda_m}\right)^{1/2}\right] \\
  & \lesssim_J  \sum_{1\leq j < m\leq J-1} \lambda_j^{-1/2} (2^k c_2)^{1/2} \lesssim_J c_1^{1/2} \tau 2^\frac{k-K}{2}.
 \end{align*}
 In summary we obtain 
 \begin{align}
  b_{k} \lesssim_1  c_2^{-2} 2^{-2k} \left(\sum_{\ell =1}^{K-k+1}  2^{-\frac{12}{5}\ell} a_{k,\ell} + 2^{-\frac{12}{5}(K-k)} a_k\right) & + c(J) \left(a_k^5 + c_1^{2/5} 2^{\frac{2}{5}(k-K)} a_k\right) \nonumber\\
  & +  c(J) c_1^{1/2} \tau 2^{\frac{k-K}{2}} \label{recursion one}
 \end{align}
 Now we give upper bounds of $a_k$ and $a_{k,\ell}$ in terms of $b_k$. First of all, According to Remark \ref{uniform scale ratio} and Lemma \ref{linearized elliptic}, we must have $w(c_2) = 0$. In other words, we have 
 \[
  \int_{-c_2}^{c_2} G^\ast(s) {\rm d} s = 0. 
 \]
 Thus
 \begin{align*}
  \left|\int_{-2^k c_2}^{2^k c_2} G^\ast(s) {\rm d} s\right| & \leq \sum_{m=0}^{k-1} \left|\int_{2^m c_2 < |s| < 2^{m+1} c_2} G^\ast(s) {\rm d} s\right|\\
  & \lesssim_1 \sum_{m=0}^{k-1} (2^m c_2)^{1/2} \|G^\ast\|_{L^2(\{s: 2^m c_2 < |s|<2^{m+1} c_2\})} \lesssim_1 \sum_{m=0}^{k-1} (2^m c_2)^{1/2} b_m. 
 \end{align*}
 Thus 
 \[
  \frac{1}{(2^k c_2)^{1/2}}\left|\int_{-2^k c_2}^{2^k c_2} G^\ast(s) {\rm d} s\right| \lesssim_1 \sum_{m=0}^{k-1} 2^\frac{m-k}{2} b_m
 \]
 The Strichartz estimates given in Corollary \ref{lemma split initial data} and Corollary \ref{lemma split data} then gives 
 \begin{align*}
  a_k & \lesssim_1  \frac{1}{(2^k c_2)^{1/2}}\left|\int_{-2^k c_2}^{2^k c_2} G^\ast(s) {\rm d} s\right| + \|G^\ast\|_{L^2(\{s: 2^k c_2<|s|<2^{k+1} c_2)} +\|\chi_{\Psi_k}\square w\|_{L^1 L^2}\\
   & \qquad + \sum_{m=k+1}^K \left(\frac{2^k c_2}{2^m c_2}\right)^{1/10} \left(\|G^\ast\|_{L^2(\{s: 2^m c_2 < |s| < 2^{m+1} c_2\})}+ \|\chi_{\Psi_m} \square w\|_{L^1 L^2}\right) \\
   & \qquad + \left(\frac{2^k c_2}{2^{K+1} c_2}\right)^{1/10} \left(\|G^\ast\|_{L^2(\{s: |s| > 2^{K+1} c_2\})} + \|\chi_{2^{K+1} c_2} \square w\|_{L^1 L^2}\right)\\ 
  & \lesssim_1 \sum_{m=0}^{k-1} 2^{\frac{m-k}{2}} b_m + \sum_{m=k}^{K} 2^{\frac{k-m}{10}} b_m + 2^\frac{k-K}{10} \tau; 
 \end{align*}
 Here we utilize the exterior energy identity \eqref{radiation residue identity} and and exterior upper bound \eqref{initial inequality w}. Similarly we may apply finite speed of propagation and deduce 
 \begin{align*}
 a_{k,\ell} \lesssim_1 \sum_{m=0}^{k-1} 2^{\frac{m-k}{2}} b_m + \sum_{m=k}^{k+\ell-1} 2^{\frac{k-m}{10}} b_m, \qquad \ell = 1,2,\cdots, K+1-k. 
 \end{align*}
 Now we collect all inequalities above: 
 \begin{align*}
  b_{k} & \leq  c_1^\ast c_2^{-2} 2^{-2k} \left(\sum_{\ell =1}^{K-k+1}  2^{-\frac{12}{5}\ell} a_{k,\ell} + 2^{-\frac{12}{5}(K-k)} a_k\right)  + c_2^\ast \left(a_k^5 + c_1^{2/5} 2^{\frac{2}{5}(k-K)} a_k\right) +  c_2^\ast c_1^{1/2} \tau 2^{\frac{k-K}{2}};\\
  b_k & \leq c_3^\ast \tau; \\
  a_k & \leq c_0^\ast \left(\sum_{m=0}^{k-1} 2^{\frac{m-k}{2}} b_m + \sum_{m=k}^{K} 2^{\frac{k-m}{10}} b_m + 2^\frac{k-K}{10} \tau\right); \\
  a_{k,\ell} & \leq c_0^\ast \left(\sum_{m=0}^{k-1} 2^{\frac{m-k}{2}} b_m + \sum_{m=k}^{k+\ell-1} 2^{\frac{k-m}{10}} b_m\right). 
 \end{align*}
 Here $c_0^\ast = c_0^\ast(1)$, $c_1^\ast = c_1^\ast(1)$ are absolute constants; and $c_2^\ast = c_2^\ast (J)$, $c_3^\ast = c_3^\ast(J)$ are two constants determined by $J$ only. The second inequality follows from \eqref{upper bound of b2}. We claim that we may choose suitable constants $c_2 = c_2(1)$ and $c_1 = c_1(J)$, $\tau_0= \tau_0(J)$, as well as another constant $\gamma = \gamma(1)$, such that if $\tau<\tau_0$ is sufficiently small, then 
 \begin{align} \label{to prove abkell}
  &b_k \leq  2^{\frac{k-K}{2}} \tau; & &a_k \leq \gamma 2^{\frac{k-K}{10}} \tau & &a_{k,\ell} \leq \gamma 2^\frac{k-K}{2} 2^{\frac{2}{5}\ell} \tau. 
 \end{align}
Please note that the upper bound of $b_k$ here immediately gives the first inequality in the conclusion of Lemma \ref{key connection}. Indeed, we may first choose the constants $\gamma, c_2 > 1$ and $c_1 < 1$ one by one such that  
 \begin{align*}
  & \gamma > 20 c_0^\ast; & & c_1^\ast c_2^{-2} \gamma < \frac{1}{10}; & \gamma c_2^\ast c_1^{2/5} < \frac{1}{10}; 
  \end{align*}
 then choose a sufficiently small constant $\tau_0 = \tau_0(J)$ such that (and that $\tau_0 < \tau(J,c_2,c_1)$, which guarantees that all the argument above holds if $\tau < \tau_0$)
\[
 16 c_2^\ast \max\left\{(c_3^\ast)^4,1\right\} \gamma^5 \tau_0^4 < \frac{1}{10};
\]
Please note that the condition $c_1$ satisfies also implies that 
\[
 c_2^\ast c_1^{1/2} < \frac{1}{10}. 
\]
Now we prove the inequalities in \eqref{to prove abkell}. Let us define 
\[
  B_k = \max\left\{0, b_k - 2^{\frac{k-K}{2}}\tau\right\}
 \]
 and 
 \begin{align*}
  &A_k = \max\left\{0, a_k - \gamma 2^{\frac{k-K}{10}}\tau \right\};& &A_{k,\ell} = \max\left\{0,a_{k,\ell} - \gamma 2^\frac{k-K}{2} 2^{\frac{2}{5}\ell} \tau\right\};
 \end{align*}
 and prove that they satisfy the inequalities: 
  \begin{align}
  B_{k} & \leq c_1^\ast c_2^{-2} 2^{-2k} \left(\sum_{\ell =1}^{K-k+1}  2^{-\frac{12}{5}\ell} A_{k,\ell} + 2^{-\frac{12}{5}(K-k)} A_k\right) + 16 c_2^\ast A_k^5 + c_2^\ast c_1^{2/5} 2^{\frac{2}{5}(k-K)} A_k; \label{inequality Bk}\\
  A_k & \leq c_0^\ast \left(\sum_{m=0}^{k-1} 2^{\frac{m-k}{2}} B_m + \sum_{m=k}^{K} 2^{\frac{k-m}{10}} B_m\right); \label{inequality Ak}\\
  A_{k,\ell} & \leq c_0^\ast \left( \sum_{m=0}^{k-1} 2^{\frac{m-k}{2}} B_m + \sum_{m=k}^{k+\ell-1} 2^{\frac{k-m}{10}} B_m\right). \label{inequality Akl} 
 \end{align}
 These three inequalities can be verified in the same manner. Since all $B_k$, $A_k$ and $A_{k,\ell}$ are nonnegative, the first inequality is trivial if $B_k = 0$. If $B_k > 0$, then we must have
 \[
  b_k = B_k + 2^\frac{k-K}{2} \tau
 \]
 Inserting this identity, as well as the inequalities 
 \begin{align*} 
  &a_k \leq A_k + \gamma 2^{\frac{k-K}{10}}\tau;& &a_{k,\ell} \leq A_{k,\ell} + \gamma 2^\frac{k-K}{2} 2^{\frac{2}{5}\ell} \tau
 \end{align*}
 into the inequality $b_k$ satisfies, we obtain 
 \begin{align*}
  B_k + 2^\frac{k-K}{2} \tau & \leq  c_1^\ast c_2^{-2} 2^{-2k} \left(\sum_{\ell =1}^{K-k+1}  2^{-\frac{12}{5}\ell} \left(A_{k,\ell} + \gamma 2^\frac{k-K}{2} 2^{\frac{2}{5}\ell} \tau\right) + 2^{-\frac{12}{5}(K-k)} \left(A_k + \gamma 2^{\frac{k-K}{10}}\tau\right)\right) \\
  & \qquad + c_2^\ast \left(\left(A_k + \gamma 2^{\frac{k-K}{10}}\tau\right)^5 + c_1^{2/5} 2^{\frac{2}{5}(k-K)} \left(A_k + \gamma 2^{\frac{k-K}{10}}\tau\right)\right) +  c_2^\ast c_1^{1/2} \tau 2^{\frac{k-K}{2}}\\
  & \leq  c_1^\ast c_2^{-2} 2^{-2k} \left(\sum_{\ell =1}^{K-k+1}  2^{-\frac{12}{5}\ell} A_{k,\ell} + 2^{-\frac{12}{5}(K-k)} A_k \right) +  c_2^\ast \left(16 A_k^5 + c_1^{2/5} 2^{\frac{2}{5}(k-K)} A_k \right) \\
  & \qquad + c_1^\ast c_2^{-2} 2^{-2k} \left(\sum_{\ell =1}^{K-k+1} \gamma 2^\frac{k-K}{2} 2^{-2\ell} \tau + \gamma 2^{-\frac{5}{2}(K-k)} \tau \right) + 16 c_2^\ast \gamma^5 2^{\frac{k-K}{2}} \tau^5 \\
  & \qquad + c_2^\ast c_1^{2/5} \gamma 2^{\frac{k-K}{2}} \tau + c_2^\ast c_1^{1/2} 2^{\frac{k-K}{2}} \tau\\
  & \leq  c_1^\ast c_2^{-2} 2^{-2k} \left(\sum_{\ell =1}^{K-k+1}  2^{-\frac{12}{5}\ell} A_{k,\ell} + 2^{-\frac{12}{5}(K-k)} A_k \right) +  c_2^\ast \left(16 A_k^5 + c_1^{2/5} 2^{\frac{2}{5}(k-K)} A_k \right) \\
  & \qquad + 2^\frac{k-K}{2} \tau \left(\gamma c_1^\ast c_2^{-2} 2^{-2k}+ \gamma c_1^\ast c_2^{-2} 2^{-2K} + 16 c_2^\ast \gamma^5 \tau_0^4 + c_2^\ast c_1^{2/5} \gamma +c_2^\ast c_1^{1/2} \right) \\
  & \leq c_1^\ast c_2^{-2} 2^{-2k} \left(\sum_{\ell =1}^{K-k+1}  2^{-\frac{12}{5}\ell} A_{k,\ell} + 2^{-\frac{12}{5}(K-k)} A_k \right) +  c_2^\ast \left(16 A_k^5 + c_1^{2/5} 2^{\frac{2}{5}(k-K)} A_k \right) \\
  & \qquad + (1/2) 2^\frac{k-K}{2} \tau.
 \end{align*}
 This verifies the inequality concerning $B_k$. Similarly if $A_k > 0$ (or $A_{k,\ell} > 0$), then we have 
 \begin{align*}
  A_k + \gamma 2^{\frac{k-K}{10}}\tau & \leq c_0^\ast \left(\sum_{m=0}^{k-1} 2^{\frac{m-k}{2}} \left(B_m + 2^\frac{m-K}{2} \tau\right) + \sum_{m=k}^{K} 2^{\frac{k-m}{10}} \left(B_m + 2^\frac{m-K}{2} \tau\right) + 2^\frac{k-K}{10} \tau\right) \\
  & \leq c_0^\ast \left(\sum_{m=0}^{k-1} 2^{\frac{m-k}{2}} B_m + \sum_{m=k}^{K} 2^{\frac{k-m}{10}} B_m + \sum_{m=0}^{k-1} 2^\frac{2m-k-K}{2}\tau + \sum_{m=k}^{K} 2^\frac{4m+k-5K}{10} \tau + 2^\frac{k-K}{10} \tau\right) \\
  & \leq c_0^\ast \left(\sum_{m=0}^{k-1} 2^{\frac{m-k}{2}} B_m + \sum_{m=k}^{K} 2^{\frac{k-m}{10}} B_m\right) + c_0^\ast 2^\frac{k-K}{2} \tau + \frac{c_0^\ast 2^\frac{k-K}{10} \tau}{1-2^{-\frac{2}{5}}} + c_0^\ast 2^\frac{k-K}{10} \tau\\
  & \leq c_0^\ast \left(\sum_{m=0}^{k-1} 2^{\frac{m-k}{2}} B_m + \sum_{m=k}^{K} 2^{\frac{k-m}{10}} B_m\right) + 7 c_0^\ast 2^\frac{k-K}{10} \tau;
 \end{align*}
 and 
 \begin{align*}
  A_{k,\ell} + \gamma 2^\frac{k-K}{2} 2^{\frac{2}{5}\ell} \tau & \leq c_0^\ast \left(\sum_{m=0}^{k-1} 2^{\frac{m-k}{2}} \left(B_m + 2^\frac{m-K}{2} \tau\right) + \sum_{m=k}^{k+\ell-1} 2^{\frac{k-m}{10}} \left(B_m + 2^\frac{m-K}{2} \tau\right)\right)\\
  & \leq c_0^\ast \left(\sum_{m=0}^{k-1} 2^{\frac{m-k}{2}} B_m + \sum_{m=k}^{k+\ell-1} 2^{\frac{k-m}{10}}B_m + \sum_{m=0}^{k-1} 2^\frac{2m-k-K}{2}\tau + \sum_{m=k}^{k+\ell-1} 2^\frac{4m+k-5K}{10} \tau\right) \\
  & \leq c_0^\ast \left(\sum_{m=0}^{k-1} 2^{\frac{m-k}{2}} B_m + \sum_{m=k}^{k+\ell-1} 2^{\frac{k-m}{10}}B_m\right) +  c_0^\ast 2^\frac{k-K}{2} \tau + \frac{c_0^\ast 2^\frac{k-K}{2} 2^{\frac{2}{5}(\ell-1)}\tau}{1-2^{-\frac{2}{5}}} \\
  & \leq c_0^\ast \left(\sum_{m=0}^{k-1} 2^{\frac{m-k}{2}} B_m + \sum_{m=k}^{k+\ell-1} 2^{\frac{k-m}{10}}B_m\right) + 6 c_0^\ast 2^\frac{k-K}{2} 2^{\frac{2}{5}\ell}\tau.
 \end{align*}
 These verify the inequalities \eqref{inequality Bk}, \eqref{inequality Ak} and \eqref{inequality Akl}. Finally we show that $B_k = A_k = A_{k,\ell} = 0$, which immediately verifies \eqref{to prove abkell}.  Indeed, let us consider 
 \[
  M =\max_{k=0,1,\cdots, K} B_k \leq \max_{k=0,1,\cdots,K} b_k \leq c_3^\ast \tau_0.
 \]
 Inserting this upper bound into \eqref{inequality Ak} and \eqref{inequality Akl}, we obtain  
 \[ 
  \max_{k,\ell} \{A_k, A_{k,\ell}\} \leq c_0^\ast \left(\frac{2^{-1/2}}{1-2^{-1/2}} + \frac{1}{1-2^{-1/10}}\right) M \leq 18 c_0^\ast M \leq \gamma M. 
 \]
 We then insert this into \eqref{inequality Bk} and obtain 
 \begin{align*}
  M  & = \max_{k} B_k \\
  & \leq \max_{k} \left[c_1^\ast c_2^{-2} 2^{-2k} \left(\sum_{\ell =1}^{K-k+1}  2^{-\frac{12}{5}\ell} \gamma M + 2^{-\frac{12}{5}(K-k)} \gamma M\right) + 16 c_2^\ast \gamma^5 M^5 + c_2^\ast c_1^{2/5} 2^{\frac{2}{5}(k-K)} \gamma M\right]\\
  & \leq \max_{k} \left[2 c_1^\ast c_2^{-2} 2^{-2k} \gamma M + 16 c_2^\ast \gamma^5 M^5 + c_2^\ast c_1^{2/5} \gamma M \right]\\
  & \leq 2 c_1^\ast c_2^{-2} \gamma M + 16 c_2^\ast \gamma^5 (c_3^\ast \tau_0)^4 M + c_2^\ast c_1^{2/5} \gamma M\\
  & \leq (2/5) M.
 \end{align*}
 As a result, we must have $M=0$, which verifies \eqref{to prove abkell} and finishes the proof of the first inequality in the conclusion. Finally the second inequality in the conclusion immediately follows from \eqref{initial inequality w}. 
 \end{proof}

Next we may further extend the domain of the approximation given in Lemma \ref{key connection}. Please note that from now on we always apply the soliton resolution theory (Proposition \ref{main tool}) with the parameter $c_2$ given in Lemma \ref{key connection}. In particular, a $J$-bubble exterior solution (see Section 3) is also defined with this parameter $c_2$. 

\begin{corollary} \label{key connection complete}
Given any positive integer $J \geq 2$ and a positive constant $c_3 < c_2$, there exists a small constant $\tau_1 = \tau_1(J,c_3)$, such that if $u$ is a $J$-bubble exterior solution to (CP1) defined in Section 3 with 
  \begin{align*}
  \tau \doteq & \left(\sup_{0<r<\lambda_{J-1}} \frac{\lambda_{J-1}}{r} \int_{-r}^r |G(s)|^2 {\rm d} s\right)^{1/2}  + \|\chi_0 v_L\|_{Y(\Rm)}  + \sup_{r>0} \frac{1}{r^{1/2}} \int_{-r}^r |G(s)| {\rm d} s  < \tau_1,
 \end{align*}
 then the error function 
 \[
  w = u - \sum_{j=1}^J \zeta_j W_{\lambda_j}(x) - \zeta_{J-1} \sqrt{3} \lambda_{J-1}^{-1/2} \varphi(x/\lambda_J) - v_L(x,t)
 \]
 and the radiation profile $G^\ast$ associated to $\vec{w}(0)$ satisfy 
 \begin{align*}
  \|G^\ast\|_{L^2(s: c_2 \lambda_J < |s| < r)} + \|\chi_{c_2 \lambda_J, r} \square w\|_{L^1 L^2} & \lesssim_{J} \left(\frac{r}{\lambda_{J-1}}\right)^{1/2}\tau, & & c_2 \lambda_J < r<  \lambda_{J-1};  \\
  \|\vec{w}(0)\|_{\mathcal{H}(\lambda_{J-1})} + \|\chi_{\lambda_{J-1}}\square w\|_{L^1 L^2} & \lesssim_J \tau; && \\
  \|G^\ast\|_{L^2(s: c_3 \lambda_J < |s| < c_2 \lambda_J)} + \|\chi_{c_3 \lambda_J, c_2\lambda_J} \square w\|_{L^1 L^2} & \lesssim_{J,c_3} \left(\frac{\lambda_J}{\lambda_{J-1}}\right)^{1/2}\tau. & &
 \end{align*}
 Here we use the same notations $v_L$, $G$, $\lambda_j$, $\zeta_j$ and $\varphi$ as in Lemma \ref{key connection}. 
\end{corollary}
\begin{proof}
 The first two inequalities immediately follows from Proposition \ref{main tool} and Lemma \ref{key connection}, as long as $\tau < \tau(J)$ is sufficiently small. Please note that the definitions of $\tau$ are different in Lemma \ref{key connection} and the current proposition. However, the values of these small $\tau$'s are always comparable, up to a constant solely determined by $J$, thanks to Proposition \ref{main tool} and Remark \ref{square estimate}. Now we prove the last inequality. Again by dilation we assume $\lambda_J=1$. We split the interval $[c_3,c_2]$ into several sub-intervals 
 \[
  c_3 = \alpha_N < \alpha_{N-1} < \cdots < \alpha_2 < \alpha_1 < \alpha_0 = c_2, \qquad \alpha_{n+1} \geq \alpha_n/2; 
 \]
 where the points $\alpha_n$'s, in particular, the number $N$ are uniqued determined by $J$ and $c_3$ and will be chosen later in the argument. It suffices to show that 
 \begin{equation} \label{induction step}
   \|G^\ast\|_{L^2(s: \alpha_n  < |s| < c_2)} + \|\chi_{\alpha_n, c_2} \square w\|_{L^1 L^2}  \lesssim_{J,n} \lambda_{J-1}^{-1/2} \tau,
 \end{equation}
 as long as $\tau < \tau(J,n)$ is sufficiently small. We conduct an induction on $n$. The inequality is trivial for $n=0$. Now we assume that \eqref{induction step} holds for a nonnegative integer $n$ and consider the case for $n+1$. We define 
 \begin{align*}
  &\Psi = \Omega_{\alpha_{n+1}, \alpha_n}& &a = \|\chi_{\Psi} w\|_{Y(\Rm)}; &
  &b = \|G^\ast\|_{L^2(s: \alpha_{n+1}<|s|<\alpha_n)} + \|\chi_{\Psi} \square w\|_{L^1 L^2};
 \end{align*}
 and let $\{\Psi_k\}_{k=0,1,2,\cdots, K}$ be the channel-like regions defined in Lemma \ref{key connection}. In addition, in order to take the advantage of finite propagation speed we define ($\ell = 0,1,\cdots,K+1$)
 \begin{align*}
  &\Psi^\ell = \left\{(x,t): \alpha_{n+1}+|t|< |x|<\alpha_n +|t|, |x|+|t|<c_2 2^\ell\right\}; & & a_\ell = \|\chi_{\Psi^\ell} w\|_{Y(\Rm)}. 
 \end{align*}
 We first give the upper bounds of $a_\ell$ and $a$ in term of $b$. First of all, we recall that 
 \[
  \int_{-c_2}^{c_2} G^\ast (s) {\rm d} s = 0,
 \]
 thus (for convenience we still use the notation $\lambda = \lambda_{J-1}$)
 \begin{align*}
  \alpha_{n+1}^{-1/2} \left|\int_{-\alpha_{n+1}}^{\alpha_{n+1}} G^\ast (s) {\rm d} s\right| & \leq \alpha_{n+1}^{-1/2} \left|\int_{\alpha_n < |s| < c_2} G^\ast (s) {\rm d} s\right| + \alpha_{n+1}^{-1/2} \left|\int_{\alpha_{n+1}<|s|<\alpha_n} G^\ast (s) {\rm d} s\right|\\
  & \lesssim_1 \left(\frac{c_2 -\alpha_n}{\alpha_{n+1}}\right)^{1/2} \|G^\ast\|_{L^2(\{s: \alpha_n < |s| < c_2\})}\\
  & \qquad  +  \left(\frac{\alpha_n-\alpha_{n+1}}{\alpha_{n+1}}\right)^{1/2} \|G^\ast\|_{L^2(\{s: \alpha_{n+1} < |s| < \alpha_n\})}\\
  & \lesssim_1  \left(\frac{2^n-1}{1/2}\right)^{1/2} c(J,n) \lambda^{-1/2} \tau + b\\
  & \lesssim_1 c(J,n) \lambda^{-1/2} \tau + b. 
 \end{align*}
 It follows from the Strichartz estimates given in Corollary \ref{lemma split initial data}/Corollary \ref{lemma split data}, the induction hypothesis and Lemma \ref{key connection} that 
 \begin{align*}
  a & \lesssim_1 \alpha_{n+1}^{-1/2} \left|\int_{-\alpha_{n+1}}^{\alpha_{n+1}} G^\ast (s) {\rm d} s\right| + \|G^\ast\|_{L^2(\{s: \alpha_{n+1} < |s| < \alpha_n\})} + \|\chi_{\Psi} \square w\|_{L^1 L^2}\\
  & \quad + \left(\frac{\alpha_n -\alpha_{n+1}}{\alpha_n}\right)^{1/10} \left(\|G^\ast\|_{L^2(s: \alpha_n < |s|<c_2)} + \|\chi_{\alpha_n, c_2}\square w\|_{L^1 L^2}\right)\\
  & \quad + \sum_{k=0}^K \left(\frac{\alpha_n -\alpha_{n+1}}{2^k c_2}\right)^{1/10} \left(\|G^\ast\|_{L^2(s: 2^k c_2 < |s|< 2^{k+1} c_2)} + \|\chi_{\Psi_k}\square w\|_{L^1 L^2}\right) \\
  & \quad + \left(\frac{\alpha_n -\alpha_{n+1}}{2^{K+1} c_2}\right)^{1/10} \left(\|G^\ast\|_{L^2(s: |s|> 2^{K+1} c_2)} + \|\chi_{2^{K+1} c_2}\square w\|_{L^1 L^2}\right) \\
  & \lesssim_1 b + c(J,n) \lambda^{-1/2} \tau + \sum_{k=0}^K \left(\frac{\alpha_n -\alpha_{n+1}}{2^k c_2}\right)^{1/10} c(J) 2^\frac{k-K}{2} \tau + \left(\frac{\alpha_n -\alpha_{n+1}}{2^{K+1} c_2}\right)^{1/10} c(J) \tau\\
  & \lesssim_1 b + c(n,J) \lambda^{-1/2} \tau + c(J) (\alpha_n-\alpha_{n+1})^{1/10} 2^{-K/10} \tau. 
 \end{align*}
Similarly we may use the finite speed of propagation to deduce 
 \begin{align*}
  a_\ell & \lesssim_1 b + c(n,J) \lambda^{-1/2} \tau + \sum_{k=0}^{\ell - 1} \left(\frac{\alpha_n -\alpha_{n+1}}{2^k c_2}\right)^{1/10} c(J) 2^\frac{k-K}{2} \tau.
 \end{align*}
 Thus we have 
 \begin{align*}
  a_0 & \lesssim_1 b + c(n,J) \lambda^{-1/2} \tau; & & \\
  a_\ell & \lesssim_1 b + c(n,J) \lambda^{-1/2} \tau + c(J) (\alpha_n-\alpha_{n+1})^{1/10} 2^{-(\ell-1)/10} 2^{(\ell-1-K)/2}\tau, & & \ell=1,2,\cdots, K+1. 
 \end{align*}
 Conversely we may also give an upper bound of $b$ in terms of $a$ and $a_\ell$. By Lemma \ref{scatter profile of nonlinear solution}, we have
 \begin{align*}
  b  \lesssim_1 \|\chi_{\alpha_{n+1}, \alpha_n} \square w\|_{L^1 L^2} \lesssim_1 I_1 + I_2 + \cdots + I_7.
 \end{align*}
 Here $I_1, I_2, \cdots, I_7$ are defined in the same manner as in the proof of Lemma \ref{key connection} 
\begin{align*}
 I_1 & = \|\chi_{\Psi} W^4 w\|_{L^1 L^2};\\
 I_2 & = c(J) \left\|\chi_{\Psi}\left(|w|^5 + \lambda^{-2} \varphi^4 |w| + |v_L|^4 |w| + \sum_{j=1}^{J-1} W_{\lambda_j}^4 |w|\right)\right\|_{L^1 L^2};\\
 I_3 & = c(J) \left\|\chi_{\Psi}\left(|v_L|^5 + \lambda^{-2} \varphi^4 |v_L|  + \sum_{j=1}^{J} W_{\lambda_j}^4 |v_L|\right)\right\|_{L^1 L^2};\\
 I_4 & = c(J) \left\|\chi_{\Psi}\left(\lambda^{-5/2} |\varphi|^5 + \lambda^{-1} W^3 |\varphi|^2 + \lambda^{-1/2} \sum_{j=1}^{J-1} W_{\lambda_{j}}^4 |\varphi|\right)\right\|_{L^1 L^2};\\
 I_5 & = \left\|\chi_{\Psi}W^4(W_\lambda - \sqrt{3}\lambda^{-1/2})\right\|_{L^1 L^2};\\
 I_6 & = c(J) \left\|\chi_{\Psi}\left(\sum_{j=1}^{J-2} W^4 W_{\lambda_j} + \sum_{j=1}^{J-1} W W_{\lambda_j}^4 + W^3 W_{\lambda_{J-1}}^2 \right)\right\|_{L^1 L^2};\\
 I_7 & = c(J) \sum_{1\leq j< m\leq J-1} \left\|\chi_{\Psi} (W_{\lambda_j}^4 W_{\lambda_m} + W_{\lambda_j} W_{\lambda_m}^4)\right\|_{L^1 L^2}.
\end{align*}
Following a similar argument to the proof of Lemma \ref{key connection} and inserting the upper bounds of $a$ and $a_\ell$, we obtain 
\begin{align*}
 I_1 & \lesssim_1 \left\|\chi_{\Psi^0} W^4 w\right\|_{L^1 L^2} + \sum_{\ell = 0}^{K} \left\|\chi_{\Psi^{\ell+1}\setminus \Psi^{\ell}} W^4 w\right\|_{L^1 L^2} + \left\|\chi_{\Psi \setminus \Psi^{K+1}} W^4 w\right\|_{L^1 L^2} \\
 & \lesssim_1 a_{0} \|\chi_{\Psi}W\|_{Y(\Rm)}^4 + \sum_{\ell = 0}^{K} \|\chi_{\Psi \setminus \Psi^{\ell}} W\|_{Y(\Rm)}^4 a_{\ell+1} + \|\chi_{\Psi \setminus \Psi^{K+1}} W\|_{Y(\Rm)}^4  a\\
 & \lesssim_1 (\alpha_n-\alpha_{n+1})^{2/5} \left(b+ c(n,J)\lambda^{-1/2} \tau\right)\\
 & \quad + \sum_{\ell = 0}^{K} (\alpha_n -\alpha_{n+1})^{2/5}(2^\ell c_2)^{-\frac{12}{5}}\left(b + c(n,J) \lambda^{-1/2} \tau + c(J) (\alpha_n-\alpha_{n+1})^{1/10} 2^{-\ell/10} 2^{(\ell-K)/2}\tau\right)\\
 & \quad +  (\alpha_n -\alpha_{n+1})^{2/5}(2^{K+1} c_2)^{-\frac{12}{5}} \left(b + c(n,J) \lambda^{-1/2} \tau + c(J) (\alpha_n-\alpha_{n+1})^{1/10} 2^{-K/10} \tau \right)\\
 & \lesssim_1 (\alpha_n-\alpha_{n+1})^{2/5} \left(b+ c(n,J)\lambda^{-1/2} \tau\right) + c(J) (\alpha_n -\alpha_{n+1})^{1/2} \sum_{\ell = 0}^K 2^{-2\ell - K/2} \tau\\
 & \lesssim_1  (\alpha_n-\alpha_{n+1})^{2/5} b + c(n,J)\lambda^{-1/2} \tau. 
 \end{align*}
 In addition, Lemma \ref{upper bound of W norm} and Lemma \ref{lemma vL bound} give 
 \begin{align*}
  \|\chi_{\Psi} W_{\lambda_j}\|_{Y(\Rm)} & \lesssim_1 \left(\frac{\alpha_n-\alpha_{n+1}}{\lambda_j}\right)^{1/10}, & &j=1,2,\cdots, J-1; \\
   \|\chi_{\Psi} v_L\|_{Y(\Rm)} & \lesssim_1 \left(\frac{\alpha_n-\alpha_{n+1}}{\lambda}\right)^{1/10} \tau. 
 \end{align*}
 We also have
 \begin{equation} \label{bound of varphi cn}
  \|\chi_\Psi \varphi\|_{Y(\Rm)} \lesssim_1 \left\|\chi_{2^{-n-1} c_2} |x|^{-1}\right\|_{Y(\Rm)} \lesssim_1 c(n),
 \end{equation} 
 We utilize these upper bounds, as well as
 \begin{align*}
  a & \lesssim_1 b + c(n,J) \lambda^{-1/2} \tau + c(J) (\alpha_n-\alpha_{n+1})^{1/10} 2^{-K/10} \tau\\
  & \lesssim_J b + c(n,J) \lambda^{-1/2} \tau + \lambda^{-1/10} \tau, 
 \end{align*}
 and deduce 
 \begin{align*}
  I_2 & \lesssim_J \|\chi_{\Psi} w\|_{Y(\Rm)}^5 + \lambda^{-2} \|\chi_{\Psi} \varphi\|_{Y(\Rm)}^4 \|\chi_{\Psi} w\|_{Y(\Rm)} + \|\chi_{\Psi} v_L\|_{Y(\Rm)}^4 \|\chi_{\Psi} w\|_{Y(\Rm)}\\
  & \qquad + \sum_{j=1}^{J-1} \|\chi_{\Psi} W_{\lambda_j}\|_{Y(\Rm)}^4 \|\chi_{\Psi} w\|_{Y(\Rm)}\\
  &\lesssim_J a^5 + \lambda^{-2} c(n) a + \left(\frac{\alpha_n-\alpha_{n+1}}{\lambda}\right)^{2/5}\tau^4 a + \left(\frac{\alpha_n-\alpha_{n+1}}{\lambda}\right)^{2/5} a \\
  & \lesssim_J b^5 + c(n,J) \lambda^{-1/2} \tau +  \lambda^{-2} c(n) b + \tau^4 b  + \lambda^{-2/5} b \\
  & \lesssim_J \left(b^4 + c(n) \lambda^{-2} + \tau^4 + \lambda^{-2/5}\right) b + c(n,J) \lambda^{-1/2} \tau.
 \end{align*}
 Lemma \ref{lemma vL bound} also gives the upper bound 
 \[
  \|\chi_{\Psi} W^4 v_L\|_{L^1 L^2} \lesssim_1 \lambda^{-1/2} \tau.
 \]
 As a result, we follow a similar argument for $I_2$ and obtain 
 \begin{align*}
  I_3 & \lesssim_J \|\chi_{\Psi} v_L\|_{Y(\Rm)}^5 + \lambda^{-2} \|\chi_{\Psi} \varphi\|_{Y(\Rm)}^4 \|\chi_{\Psi} v_L\|_{Y(\Rm)} + \sum_{j=1}^{J-1} \|\chi_{\Psi} W_{\lambda_j}\|_{Y(\Rm)}^4 \|\chi_{\Psi} v_L\|_{Y(\Rm)} \\
  & \qquad + \left\|\chi_{\Psi} W^4 v_L\right\|_{L^1 L^2}\\
  & \lesssim_J \left(\frac{\alpha_n -\alpha_{n+1}}{\lambda}\right)^{1/2} \tau^5 + \lambda^{-2} c(n) \tau + \left(\frac{\alpha_n -\alpha_{n+1}}{\lambda}\right)^{1/2} \tau + \lambda^{-1/2} \tau \\
  & \lesssim_{J,n} \lambda^{-1/2} \tau. 
 \end{align*}
  Next we recall \eqref{bound of varphi cn}, apply Corollary \ref{w lambda varphi} and obtain 
  \begin{align*}
  I_4 & \lesssim_J \lambda^{-5/2} \|\chi_{\Psi} \varphi\|_{Y(\Rm)}^5 + \lambda^{-1} \|\chi_{\Psi} W\|_{Y(\Rm)}^3  \|\chi_{\Psi} \varphi\|_{Y(\Rm)}^2 + \lambda^{-1/2} \sum_{j=1}^{J-1} \left\|\chi_{\Psi} W_{\lambda_j}^4 \varphi\right\|_{L^1 L^2}\\
  & \lesssim_{J,n} \lambda^{-5/2} + \lambda^{-1} + \lambda^{-1/2} \sum_{j=1}^{J-1} \lambda_j^{-1} (\alpha_n - \alpha_{n+1})^{1/2}\\
  & \lesssim_{J,n} \lambda^{-1} \lesssim_{J,n} \lambda^{-1/2} \tau.
 \end{align*}
 The estimates of $I_5$ and $I_6$ immediately follow from Lemma \ref{estimate W J lambda} and Lemma \ref{L1L2 channel estimate mixed}: 
 \begin{align*}
  I_5 & \lesssim_1 (\alpha_n - \alpha_{n+1})^{1/2} \lambda^{-5/2} \ln \lambda \lesssim_1 \lambda^{-1} \lesssim_J \lambda^{-1/2} \tau;\\
  I_6 & \lesssim_J \sum_{j=1}^{J-2} \lambda_j^{-1/2} (\alpha_n -\alpha_{n+1})^{1/2} + \sum_{j=1}^{J-1} \lambda_j^{-1} (\alpha_n-\alpha_{n+1})^{1/2} + \lambda^{-1} (\alpha_n - \alpha_{n+1})^{1/2} \\
  & \lesssim_J \lambda^{-1/2} \tau + \lambda^{-1} \lesssim_J \lambda^{-1/2} \tau. 
 \end{align*}
 Finally we combine Lemma \ref{L1L2 channel estimate mixed} and the dilation invariance to deduce 
 \begin{align*}
  I_7 & \lesssim_J \sum_{1\leq j < m\leq J-1} \left[\left(\frac{\lambda_j}{\lambda_m}\right)^{-1} \left(\frac{\alpha_n-\alpha_{n+1}}{\lambda_m}\right)^{1/2} + \left(\frac{\lambda_j}{\lambda_m}\right)^{-1/2}\left(\frac{\alpha_n-\alpha_{n+1}}{\lambda_m}\right)^{1/2}\right] \\
  & \lesssim_J  \sum_{1\leq j < m\leq J-1} \lambda_j^{-1/2} (\alpha_n - \alpha_{n+1})^{1/2} \lesssim_J \lambda^{-1/2} \tau.
 \end{align*}
In summary, we obtain an inequality 
\begin{align} \label{recursion two}
 b \leq c_1^\ast \left[(\alpha_n-\alpha_{n+1})^{2/5} + b^4 + c(n) \lambda^{-2} + \tau^4 + \lambda^{-2/5}\right]b + c(n,J) \lambda^{-1/2 }\tau. 
\end{align}
Here $c_1^\ast$ is a positive constant solely determined by $J$. We may choose $\alpha_n$ such that 
\[
 c_1^\ast (\alpha_n-\alpha_{n+1})^{2/5} < 1/4.
\]
In addition, we may recall Remark \ref{remark delta varphi} and give an upper bound of $b$ 
\begin{align*}
 b & \lesssim_1 \|\chi_{\Psi} \square w\|_{L^1 L^2} \\
 & \lesssim_1 \left\|\chi_{\Psi} \left(F(u) - \sum_{j=1}^J \zeta_j F(W_{\lambda_j})\right)\right\|_{L^1 L^2} + \lambda^{-1/2} \|\chi_{\Psi} \Delta \varphi\|_{L^1 L^2} \\
 & \lesssim_J \tau.
\end{align*}
As a result, when $\tau < \tau(n,J)$ is sufficiently small, we must have 
\[
 c_1^\ast \left[ b^4 + c(n) \lambda^{-2} + \tau^4 + \lambda^{-2/5}\right] < 1/4.
\]
Thus the terms with $b$ in the right hand side of \eqref{recursion two} can be absorbed by the left hand side, which leads to  
\[
 b \lesssim_{n,J} \lambda^{-1/2} \tau.
\]
A combination of this with the induction hypothesis verifies \eqref{induction step} for $n+1$. This completes the proof. 
\end{proof} 

\begin{lemma} \label{J bubble upper bound inner}
 Given an integer $J\geq 2$, there exists a small constant $\tau_2 = \tau_2(J)$ and a large constant $M = M(J)$, such that if $u$ is a radial $J$-bubble exterior solution to (CP1) with 
   \begin{align*}
  \tau \doteq & \left(\sup_{0<r<\lambda_{J-1}} \frac{\lambda_{J-1}}{r} \int_{-r}^r |G(s)|^2 {\rm d} s\right)^{1/2}  + \|\chi_0 v_L\|_{Y(\Rm)}  + \sup_{r>0} \frac{1}{r^{1/2}} \int_{-r}^r |G(s)| {\rm d} s  < \tau_2,
 \end{align*}
then the error function 
 \[
  w_\ast = u - \sum_{j=1}^J \zeta_j W_{\lambda_j}(x) - v_L(x,t)
 \]
 satisfies 
 \[
  \sup_{0<r<c_2 \lambda_J} r^{1/2} |w_\ast (r,0)| \leq M \left(\frac{\lambda_J}{\lambda_{J-1}}\right)^{1/2}. 
 \]
 Here we use the same notations $c_2$, $v_L$, $G$, $\zeta_j$, $\lambda_j$ as in Lemma \ref{key connection}. 
\end{lemma}
\begin{proof}
 Let $c_3 = c_3 (J)< 1$ be a constant to be determined later. Without loss of generality we still assume $\lambda_J=1$ and use the notation $\lambda = \lambda_{J-1}$. We apply Lemma \ref{key connection}, as well as Corollary \ref{key connection complete}, and obtain 
 \begin{align}
  \|G^\ast\|_{L^2(s: 2^k c_2 < |s| < 2^{k+1} c_2)} + \|\chi_{2^k c_2, 2^{k+1} c_2} \square w\|_{L^1 L^2} & \lesssim_{J} 2^\frac{k-K}{2} \tau, & & k=0,1,\cdots,K;  \nonumber\\
  \|\vec{w}(0)\|_{\mathcal{H}(c_1 \lambda)} + \|\chi_{c_1 \lambda}\square w\|_{L^1 L^2} & \lesssim_J \tau; && \nonumber\\
  \|G^\ast\|_{L^2(s: c_3  < |s| < c_2)} + \|\chi_{c_3, c_2} \square w\|_{L^1 L^2} & \lesssim_{J,c_3} \lambda^{-1/2} \tau; & & \label{G up bound}
 \end{align}
 as long as $\tau < \tau(J,c_3)$ is sufficiently small. Here $G^\ast$ is the radiation profile of $\vec{w}(0)$, which is defined in Lemma \ref{key connection}. The constant $c_1$ and positive integer $K$ are also defined in Lemma \ref{key connection}. We observe 
 \[
  w_\ast = w + \zeta_{J-1} \sqrt{3} \lambda^{-1/2} \varphi(x)
 \]
 It follows from Lemma \ref{linearized elliptic} and Remark \ref{remark delta varphi} that the radiation profile $G_\ast$ of $\vec{w}_\ast(0)$ and $\square w_\ast$ satisfy
 \begin{align}
  \|G_\ast-G^\ast\|_{L^2(s: |s|>c_3)} & \lesssim_{c_3} \lambda^{-1/2}; \label{G diff bound}\\
  \|\chi_{c_3}(\square w_\ast - \square w)\|_{L^1 L^2} & \lesssim_1 \lambda^{-1/2}. \nonumber
 \end{align}
 We define 
 \begin{align*}
  &a = \|\chi_{0, c_3} w_\ast\|_{Y(\Rm)}; &
  &b = \|G_\ast\|_{L^2(-c_3,c_3)} + \|\chi_{0, c_3} \square w_\ast\|_{L^1 L^2};
 \end{align*}
 and ($\ell = 0,1,2,\cdots,K+1$)
 \begin{align*}
  &\Psi^\ell = \left\{(x,t): |t|< |x|<c_3 +|t|, |x|+|t|<c_2 2^\ell\right\}; & & a_\ell = \|\chi_{\Psi^\ell} w_\ast\|_{Y(\Rm)}. 
 \end{align*}
 Let us first give a rough upper bound of $b$. We may apply Proposition \ref{main tool} and Remark \ref{square estimate} to deduce 
 \begin{align} \label{first upper bound of b3}
  b \lesssim_1 \|\vec{w}_\ast(0)\|_{\dot{H}^1\times L^2} + \|\chi_0 \square w_\ast\|_{L^1 L^2} \lesssim_J \tau. 
 \end{align}
 Next we give more dedicate upper bounds of $b$ and $a$. In order to give an upper bound of $a$, we let $w_1$, $w_2$ be the radial free wave with radiation profiles $G_1, G_2$ below, respectively,
 \begin{align*}
  & G_1(s) = \left\{\begin{array}{ll} G_\ast(s), & |s|<c_3; \\ G^\ast(s), & |s|>c_3; \end{array}\right. & 
  & G_2(s) = \left\{\begin{array}{ll} 0, & |s|<c_3; \\ G_\ast(s) - G^\ast(s), & |s|>c_3; \end{array}\right.
 \end{align*}
 and $w_3$, $w_4$ be the solution to the following wave equation with zero initial data, respectively, 
 \begin{align*}
  & \square w_3 = \left\{\begin{array}{ll} 0, & |x|<|t|; \\ \square w_\ast, & (x,t)\in \Omega_{0,c_3}; \\ \square w, & (x,t) \in \Omega_{c_3}; \end{array}\right. & 
  \square w_4 = \left\{\begin{array}{ll} 0, & |x|<|t|; \\ 0, & (x,t)\in \Omega_{0,c_3}; \\ \square w_\ast - \square w, & (x,t) \in \Omega_{c_3}. \end{array}\right.
 \end{align*}
 It follows from the finite speed of propagation that 
 \[
  w_\ast(x,t) = w_1(x,t) + w_2(x,t) + w_3(x,t) + w_4(x,t), \qquad (x,t) \in \Psi \doteq \Omega_{0,c_3}. 
 \]
  We then apply Remark \ref{remark split initial data} on $w_1$, Corollary \ref{lemma split data} on $w_3$, and the regular Strichartz estimates on $w_2$, $w_4$ to deduce
 \begin{align*}
  a & \leq \|\chi_\Psi w_1\|_{Y(\Rm)} + \|\chi_\Psi w_2\|_{Y(\Rm)} + \|\chi_\Psi w_3\|_{Y(\Rm)} + \|\chi_\Psi w_4\|_{Y(\Rm)} \\
  & \lesssim_1 \|G_\ast\|_{L^2(-c_3,c_3)} + \|G^\ast\|_{L^2(\{s: c_3<|s|<c_2)\}} + \sum_{k=0}^{K} \left(\frac{c_3}{2^k c_2}\right)^{1/10} \|G^\ast\|_{L^2(\{s: 2^k c_2 < |s| < 2^{k+1} c_2\})} \\
  & \qquad + \left(\frac{c_3}{2^{K+1} c_2}\right)^{1/10} \|G^\ast\|_{L^2(\{s: |s| > 2^{K+1} c_2\})} + \|G_\ast - G^\ast\|_{L^2(\{s: |s|>c_3\})}\\
  & \qquad + \|\chi_\Psi \square w_\ast\|_{L^1 L^2} + \|\chi_{c_3,c_2} \square w\|_{L^1 L^2} + \sum_{k=0}^K \left(\frac{c_3}{2^k c_2}\right)^{1/10} \|\chi_{2^k c_2, 2^{k+1} c_2} \square w\|_{L^1 L^2}\\
  & \qquad + \left(\frac{c_3}{2^{K+1} c_2}\right)^{1/10} \|\chi_{2^{K+1} c_2} \square w\|_{L^1 L^2} + \|\chi_{c_3} (\square w_\ast - \square w)\|_{L^1 L^2}\\
  & \lesssim_1 b + c(J,c_3)\lambda^{-1/2} \tau + \sum_{k=0}^K  \left(\frac{c_3}{2^k c_2}\right)^{1/10} c(J) 2^\frac{k-K}{2} \tau + \left(\frac{c_3}{2^{K+1} c_2}\right)^{1/10} c(J) \tau + c(c_3) \lambda^{-1/2} \\
  & \lesssim_1 b+ c(J,c_3) \lambda^{-1/2} + c(J) c_3^{1/10} \lambda^{-1/10} \tau. 
 \end{align*}
 Here we use $2^K \simeq_J \lambda$. Similarly we may incorporate the finite propagation speed and obtain 
 \begin{align*}
  a_\ell  \lesssim_1 b + c(J,c_3)\lambda^{-1/2} \tau + \sum_{k=0}^{\ell-1}  \left(\frac{c_3}{2^k c_2}\right)^{1/10} c(J) 2^\frac{k-K}{2} \tau + c(c_3) \lambda^{-1/2}
  \end{align*}
 Thus 
 \begin{align*}
 a_0 & \lesssim_1 b + c(J,c_3)\lambda^{-1/2}; & & \\
 a_\ell & \lesssim_1 b + c(J,c_3) \lambda^{-1/2} + c(J) c_3^{1/10} 2^{-\frac{\ell-1}{10}} 2^\frac{\ell-1-K}{2} \tau, & & \ell = 1,2,\cdots, K+1
 \end{align*}
 Now we give an upper bound of $b$. We apply Lemma \ref{scatter profile of nonlinear solution} on $w_\ast$ and obtain 
 \begin{align*}
  b & \lesssim_1 \|\chi_{0,c_3}\square w_\ast\|_{L^1 L^2} \\
  & \lesssim_1 \left\|\chi_{0,c_3}\left(F\left(w_\ast + v_L + \sum_{j=1}^J \zeta_j W_{\lambda_j}\right) - \sum_{j=1}^J \zeta_j F(W_{\lambda_j})\right)\right\|_{L^1 L^2} \\
  & \lesssim_1 I_1 + I_2 + I_3 + I_4;
 \end{align*}
 with
 \begin{align*}
  I_1 & = c(J) \left\|\chi_{0,c_3} W^4 w_\ast\right\|_{L^1 L^2}; \\
  I_2 & = c(J) \left\|\chi_{0,c_3} \left(|w_\ast|^5 + |v_L|^4 |w_\ast| + \sum_{j=1}^{J-1} W_{\lambda_j}^4 |w_\ast|\right)\right\|_{L^1 L^2};\\
  I_3 & = c(J) \left\|\chi_{0,c_3}\left(|v_L|^5 + \sum_{j=1}^J W_{\lambda_j}^4 |v_L|\right)\right\|_{L^1 L^2}; \\
  I_4 & = c(J) \left\|\chi_{0,c_3} \sum_{1\leq j< m\leq J}\left(W_{\lambda_j}^4 W_{\lambda_m} + W_{\lambda_j} W_{\lambda_m}^4\right)\right\|_{L^1 L^2}. 
 \end{align*}
 Following the same argument as in proof of Corollary \ref{key connection complete}, we obtain 
 \begin{align*}
 I_1 & \lesssim_J \left\|\chi_{\Psi^{0}} W^4 w_\ast\right\|_{L^1 L^2} + \sum_{\ell = 0}^{K} \left\|\chi_{\Psi^{\ell+1}\setminus \Psi^{\ell}} W^4 w_\ast\right\|_{L^1 L^2} + \left\|\chi_{\Omega_{0,c_3} \setminus \Psi^{K+1}} W^4 w_\ast\right\|_{L^1 L^2} \\
 & \lesssim_J a_{0} \|\chi_{0,c_3}W\|_{Y(\Rm)}^4 + \sum_{\ell = 0}^{K} \|\chi_{\Omega_{0,c_3} \setminus \Psi_{\ell}} W\|_{Y(\Rm)}^4 a_{\ell+1} + \|\chi_{\Omega_{0,c_3} \setminus \Psi_{K+1}} W\|_{Y(\Rm)}^4  a\\
 & \lesssim_J c_3^{2/5} a_{0}  + \sum_{\ell = 0}^{K} c_3^{2/5} 2^{-12\ell/5} a_{\ell+1} + c_3^{2/5} 2^{-12K/5}  a\\
 & \lesssim_J c_3^{2/5} \left(b+c(J,c_3) \lambda^{-1/2}\right) + \sum_{\ell = 0}^{K} c_3^{2/5} 2^{-12\ell/5}\left(b + c(J,c_3) \lambda^{-1/2} + c(J) c_3^{1/10} 2^{-\frac{\ell}{10}} 2^\frac{\ell-K}{2} \tau\right)\\
  & \qquad + c_3^{2/5} 2^{-12K/5} \left(b+ c(J,c_3) \lambda^{-1/2} + c(J) c_3^{1/10} \lambda^{-1/10} \tau\right)\\
 & \lesssim_J c_3^{2/5} b + c(c_3,J) \lambda^{-1/2}. 
 \end{align*}
 In addition, we apply Lemma \ref{upper bound of W norm}, Lemma \ref{lemma vL bound} and obtain
 \begin{align*}
  I_2 & \lesssim_J \|\chi_{0,c_3} w_\ast\|_{Y(\Rm)}^5 + \|\chi_{0,c_3} v_L\|_{Y(\Rm)}^4 \|\chi_{0,c_3} w_\ast\|_{Y(\Rm)} + \sum_{j=1}^{J-1} \|\chi_{0,c_3} W_{\lambda_j}\|_{Y(\Rm)}^4 \|\chi_{0,c_3} w_\ast\|_{Y(\Rm)}\\
  & \lesssim_J a^5 + \left(\frac{c_3}{\lambda}\right)^{2/5} \tau^4 a + \left(\frac{c_3}{\lambda}\right)^{2/5} a \\
 & \lesssim_J \left[b^5 + c(c_3,J) \lambda^{-1/2}\right] + \left[\tau^4 b + c(c_3,J) \lambda^{-1/2} \tau^4\right]  + \left[(c_3/\lambda)^{2/5} b + c(c_3,J) \lambda^{-1/2}\right]\\
 & \lesssim_J \left(b^4 + \tau^4 + c_3^{2/5} \lambda^{-2/5} \right) b + c(c_3,J) \lambda^{-1/2}\\
 & \lesssim_J \left(\tau^4 + c_3^{2/5}\right) b + c(c_3,J) \lambda^{-1/2}. 
 \end{align*}
 Here we utilize \eqref{first upper bound of b3} in the last step. By Lemma \ref{lemma vL bound} we also have
 \begin{align*}
  I_3 & \lesssim_J \|\chi_{0,c_3} v_L\|_{Y(\Rm)}^5 + \sum_{j=1}^{J-1} \|\chi_{0,c_3} W_{\lambda_j}\|_{Y(\Rm)}^4 \|\chi_{0,c_3} v_L\|_{Y(\Rm)} + \left\|\chi_{0,c_3} W^4 v_L\right\|_{L^1 L^2}\\
  &\lesssim_J \left(\frac{c_3}{\lambda}\right)^{1/2} \tau^5 + \left(\frac{c_3}{\lambda}\right)^{1/2} \tau + c_3^{1/2} \lambda^{-1/2} \tau \\
   & \lesssim_J \lambda^{-1/2}. 
 \end{align*}
 Finally we combine Lemma \ref{L1L2 channel estimate mixed} and the dilation invariance to deduce 
 \begin{align*}
  I_4 & \lesssim_J \sum_{1\leq j < m\leq J} \left[\left(\frac{\lambda_j}{\lambda_m}\right)^{-1} \left(\frac{c_3}{\lambda_m}\right)^{1/2} + \left(\frac{\lambda_j}{\lambda_m}\right)^{-1/2}\left(\frac{c_3}{\lambda_m}\right)^{1/2}\right] \\
  & \lesssim_J  \sum_{1\leq j < m\leq J} c_3^{1/2} \lambda_j^{-1/2}  \lesssim_J \lambda^{-1/2}. 
 \end{align*}
 In summary, we have 
 \begin{equation} \label{b recurrence 3} 
  b \leq c^\ast \left(c_3^{2/5} + \tau^4\right) b + c(c_3,J) \lambda^{-1/2}. 
 \end{equation} 
 Here $c^\ast$ is a constant solely determined by $J$. We choose a small constant $c_3 = c_3(J)$ such that $c^\ast c_3^{2/5} < 1/4$ and let $\tau_2 =  \tau_2 (J)$ be sufficiently small such that $c^\ast \tau_2^4 < 1/4$ (and that all the argument above works if $\tau < \tau_2$, please note that now $\tau(J,c_3)$ depends on $J$ only). As a result, if $\tau < \tau_2$, then it immediately follows from \eqref{b recurrence 3} that 
 \[
  b \lesssim_{J} \lambda^{-1/2}. 
 \]
 A combination of this upper bound with the upper bounds \eqref{G up bound} and \eqref{G diff bound} then yields 
 \[
  \|G_\ast\|_{L^2(-c_2,c_2)} \lesssim_J \lambda^{-1/2},
 \]
 which immediately gives the conclusion by the following estimates
 \[
   r^{1/2} |w_\ast (r,0)| = \frac{1}{r^{1/2}} \left|\int_{-r}^r G_\ast (s) {\rm d} s\right| \lesssim_1 \|G_\ast\|_{L^2(-r,r)} \lesssim_J \lambda^{-1/2}, \quad \forall r\in (0,c_2). 
 \]
\end{proof}

\section{Proof of main theorem}

In this section we first give the radiation concentration property of $J$-bubble solutions and then prove the main theorem of this work as an application. 

\subsection{Radiation concentration}

\begin{proposition} \label{prop concentration}
 Given a positive integer $J$, there exists a small constant $\tau_3=\tau_3(J)>0$, such that any radial $J$-bubble exterior solution $u$ to (CP1) must satisfy 
\[
 \tau \doteq \left(\sup_{0<r<\lambda_{J-1}} \frac{\lambda_{J-1}}{r} \int_{-r}^r |G(s)|^2 {\rm d} s\right)^{1/2}  + \|\chi_0 v_L\|_{Y(\Rm)}  + \sup_{r>0} \frac{1}{r^{1/2}} \int_{-r}^r |G(s)| {\rm d} s  \geq \tau_3.
\]
Here $v_L$ is the radiation part of $u$; $G$ is the corresponding radiation profile of $v_L$; $\lambda_{J-1}$ is the size of $(J-1)$-th bubble given by Proposition \ref{main tool}.
\end{proposition}
\begin{proof}
 Let $\tau_2 = \tau_2(J)$ and $M = M(J)$ be the constants given in Lemma \ref{J bubble upper bound inner}. Let us consider the approximated solution $w$ and $w_\ast$ given in Corollary \ref{key connection complete} and Lemma \ref{J bubble upper bound inner} respectively. We recall the asymptotic behaviour $|\varphi(r)|\gtrsim_1 r^{-1}$ near the zero (see Lemma \ref{linearized elliptic}) and obtain 
 \[
  r^{1/2} |w(r,0) - w_\ast(r,0)| = r^{1/2} \sqrt{3} \lambda_{J-1}^{-1/2} |\varphi(r/\lambda_J)| \gtrsim_1 \left(\frac{r}{\lambda_J}\right)^{-1/2} \left(\frac{\lambda_{J}}{\lambda_{J-1}}\right)^{1/2},
 \]
 as long as $r/\lambda_J \ll 1$ is sufficiently small. Thus we may choose a small constant $c_4 = c_4(J)< c_2$ such that 
 \begin{equation} \label{diff w wast last}
  (c_4 \lambda_J)^{1/2} |w(c_4 \lambda_J,0) - w_\ast(c_4 \lambda_J,0)| > 2 M(J) \left(\frac{\lambda_{J}}{\lambda_{J-1}}\right)^{1/2}. 
 \end{equation} 
 Now let $\tau_1 = \tau_1(J,c_4)$ is the constant given in Corollary \ref{key connection complete}, which is unique determined by $J$. If $u$ is a $J$-bubble exterior solution with $\tau < \min\{\tau_1, \tau_2\}$, then it follows from Corollary \ref{key connection complete} and Lemma \ref{J bubble upper bound inner} that
 \begin{align}
  \|G^\ast\|_{L^2(s: c_4 \lambda_J < |s| < c_2 \lambda_J)} & \leq C \left(\frac{\lambda_{J}}{\lambda_{J-1}}\right)^{1/2} \tau;\nonumber\\
  (c_4 \lambda_J)^{1/2} |w_\ast(c_4 \lambda_J,0)| & \leq M(J) \left(\frac{\lambda_{J}}{\lambda_{J-1}}\right)^{1/2}. \label{wast last}
 \end{align}
 Here $C=C(J,c_4)$ is a constant solely determined by $J$; $G^\ast$ is the radiation profile of $\vec{w}(0)$. Now we may apply the explicit formula of initial data in term of radiation profile to deduce 
 \begin{align*}
  (c_4 \lambda_J) w(c_4 \lambda_J,0) = (c_2 \lambda_J) w(c_2 \lambda_J,0) - \int_{c_4 \lambda_J <|s|<c_2 \lambda_J} G^\ast (s) {\rm d} s. 
 \end{align*}
 Since $w(c_2 \lambda_J,0) = w_\ast(c_2 \lambda_J,0) = 0$, we have 
 \begin{align*}
  \left|(c_4 \lambda_J) w(c_4 \lambda_J,0)\right| & \leq \int_{c_4 \lambda_J <|s|<c_2 \lambda_J} |G^\ast (s)| {\rm d} s\\
  & \leq \sqrt{2} (c_2-c_4)^{1/2} \lambda_J^{1/2} \|G^\ast(s)\|_{L^2(\{s: c_4 \lambda_J < |s| < c_2 \lambda_J\})}\\
  & \leq \sqrt{2} C (c_2-c_4)^{1/2} \lambda_J^{1/2} \left(\frac{\lambda_{J}}{\lambda_{J-1}}\right)^{1/2} \tau. 
 \end{align*}
 This implies 
 \[
  (c_4 \lambda_J)^{1/2} \left|w(c_4 \lambda_J,0)\right| \leq \sqrt{2} C \left(\frac{c_2-c_4}{c_4}\right)^{1/2}\left(\frac{\lambda_{J}}{\lambda_{J-1}}\right)^{1/2} \tau. 
 \]
 A combination of this with \eqref{diff w wast last} and \eqref{wast last} yields
 \[
   M(J) \left(\frac{\lambda_{J}}{\lambda_{J-1}}\right)^{1/2} \leq \sqrt{2} C \left(\frac{c_2-c_4}{c_4}\right)^{1/2}\left(\frac{\lambda_{J}}{\lambda_{J-1}}\right)^{1/2} \tau,
 \]
 which implies 
 \[
  \tau \geq  \frac{M(J)}{\sqrt{2} C} \left(\frac{c_4}{c_2-c_4}\right)^{1/2}. 
 \]
 Here the lower bound depends on $J$ only. In summary, any $J$-bubble exterior solution satisfies 
 \[
  \tau \geq \tau_3(J) \doteq \min\left\{\tau_1, \tau_2, \frac{M(J)}{\sqrt{2} C} \left(\frac{c_4}{c_2-c_4}\right)^{1/2} \right\},
 \] 
 which completes the proof. 
\end{proof}

\subsection{Proof of the main theorem}

The rest of this section is devoted to the proof of the main theorem. According to the soliton resolution theorem given in Duyckaerts-Kenig-Merle \cite{se}, it suffices to exclude all situations of soliton resolution with two or more bubbles. The proof of the global case and the type II blow-up case is similar. We consider the global case first.

\paragraph{The global case} If the soliton resolution of a global solution $u$ came with $J$ bubbles for some positive integer $J\geq 2$, we would give a contradiction. Let us first show the existence of nonlinear radiation profile of $u$. Let $R>0$ be a large number such that $\|\vec{u}(0)\|_{\mathcal{H}(R)} \ll 1$. By the small data theory, a standard cut-off technique and finite speed of propagation, there exists a finite-energy free wave $u_{L}^-$ such that 
\begin{align*}
 \lim_{t\rightarrow -\infty} \int_{|x|>R+|t|} |\nabla_{t,x}(u-u_L^-)(x,t)|^2 {\rm d} x =0. 
\end{align*}
Let $G_-$ be the radiation profile of $u_L^-$ in the negative time direction. This immediately gives the (nonlinear) radiation profile in the negative time direction 
\[
 \lim_{t\rightarrow -\infty} \int_{R-t}^\infty \left(\left|G_-(r+t) - r u_t(r,t)\right|^2 +  \left|G_-(r+t) - r u_r(r,t)\right|^2\right) {\rm d} r = 0. 
\]
Next we consider the positive time direction. According to Lemma 3.7 of Duyckaerts-Kenig-Merle \cite{se}, there exists a finite-energy free wave $u_L$, such that
\[
 \lim_{t\rightarrow +\infty} \int_{|x|>t-A} |\nabla_{t,x}(u-u_L)(x,t)|^2 {\rm d} x =0, \qquad \forall A \in \Rm. 
\]
Thus the radiation profile $G_+$ of $u_L$ in the positive time direction becomes the (nonlinear) radiation profile of $u$, i.e. 
\[
 \lim_{t\rightarrow +\infty} \int_{t-A}^\infty \left(\left|G_+(r-t) - r u_t(r,t)\right|^2 +  \left|G_+(r-t) + r u_r(r,t)\right|^2\right) {\rm d} r = 0, \qquad \forall A \in \Rm. 
\]
Let us consider the time-translated solution $u(\cdot, \cdot+t)$ for $t>R$, which is defined at least outside the main light cone, whose (nonlinear) radiation profiles can be given by 
\begin{align} \label{relationship G Gt}
 &G_{t,+}(s) = G_+(s-t), \quad s > 0;& & G_{t,-}(s) = G_-(s+t), \quad s>0. 
\end{align}
We let $v_{t,L}$ be the free wave whose radiation profiles in two time directions are equal to $G_{t,+}$ and $G_{t,-}$ for $s>0$, respectively. It is not difficult to see that $v_{t,L}$ is exactly the asymptotically equivalent free wave of $u$. We claim that the following limit holds 
\begin{equation} \label{global part tau bound}
 \lim_{t\rightarrow +\infty} \left(\|\chi_0 v_{t,L}\|_{Y(\Rm)} + \sup_{r>0} \frac{1}{r^{1/2}} \int_{0}^r \left(|G_{t,+}(s)| +|G_{t,-}(s)| \right) {\rm d} s\right) = 0. 
\end{equation} 
In fact, for any small constant $\varepsilon > 0$, we may find an interval $[a,b] \subset \Rm$ and a large number $T_0 \geq R$, such that  
\begin{align*}
 &\|G_+\|_{L^2(\Rm\setminus [a,b])} < \varepsilon; & & \|G_-\|_{L^2([T_0,+\infty))} < \varepsilon. 
\end{align*}
Now let us consider very large time $t \gg \max\{T_0,-a\}$. According to Lemma \ref{calculation 1}, we have 
\[
 \|\chi_0 v_{t,L}\|_{Y(\Rm)} \lesssim_1 \varepsilon + \left(\frac{b-a}{t+a}\right)^{1/2} \|G_+\|_{L^2([a,b])}.
\]
In addition, we also have 
\begin{align*}
 \frac{1}{r^{1/2}} \int_{0}^r \left(|G_{t,+}(s)| +|G_{t,-}(s)| \right) {\rm d} s = \frac{1}{r^{1/2}} \left(\int_{-t}^{-t+r} |G_+(s)| {\rm d} s + \int_{t}^{t+r} |G_-(s)| {\rm d} s\right). 
\end{align*}
If $r<t+a$, then the interval $[-t,-t+r]$ does not intersects with $[a,b]$, thus the Cauchy-Schwarz inequality gives 
\[
 \frac{1}{r^{1/2}} \int_{0}^r \left(|G_{t,+}(s)| +|G_{t,-}(s)| \right) {\rm d} s  \lesssim_1 \varepsilon. 
\]
On the other hand, if $r > t+a$, then we have 
\[ 
 \frac{1}{r^{1/2}} \int_{0}^r \left(|G_{t,+}(s)| +|G_{t,-}(s)| \right) {\rm d} s  \lesssim_1 \varepsilon + \frac{1}{r^{1/2}} \int_a^b |G_+(s)| {\rm d} s  
\]
In summary, we have 
\[
 \sup_{r>0} \left(\frac{1}{r^{1/2}} \int_{0}^r \left(|G_{t,+}(s)| +|G_{t,-}(s)| \right) {\rm d} s\right)  \lesssim_1 \varepsilon + \left(\frac{b-a}{t+a}\right)^{1/2} \|G_+\|_{L^2([a,b])}. 
\]
As a result, the following inequality holds 
\[
  \limsup_{t\rightarrow +\infty} \left(\|\chi_0 v_{t,L}\|_{Y(\Rm)} + \sup_{r>0} \frac{1}{r^{1/2}} \int_{0}^r \left(|G_{t,+}(s)| +|G_{t,-}(s)| \right) {\rm d} s\right) \lesssim_1 \varepsilon. 
\]
Since $\varepsilon>0$ is arbitrary, the limit \eqref{global part tau bound} immediately follows. In addition, our assumption that the soliton resolution of $u$ comes with $J$ bubbles implies that the time translated solution $u(\cdot,\cdot+t)$ is a $J$-bubble exterior solution as defined at the end of Section 3 for sufficiently large time. In fact, almost orthogonality of decoupled bubbles imply that the following energy estimates hold as long as $t$ is sufficiently large: 
\begin{itemize}
 \item If the soliton resolution of $u(\cdot,\cdot +t)$ given by Proposition \ref{main tool} (with $n=J+1$) is incomplete(i.e. in case b), then 
 \[
  \left\|\vec{u}(t) - \vec{v}_{t,L}(0)\right\|_{\dot{H}^1\times L^2}^2 > J \|W\|_{\dot{H}^1}^2 + \frac{1}{2} \|W\|_{\dot{H}^1(\{x: |x|>c_2\})}^2. 
 \]
 \item If the soliton resolution of $u(\cdot,\cdot +t)$ given by Proposition \ref{main tool} comes with exactly $J_1$ bubbles for some $J_1\in \{0,1,\cdots,J\}$, then 
 \[
  \left|\left\|\vec{u}(t) - \vec{v}_{t,L}(0)\right\|_{\dot{H}^1\times L^2}^2 - J_1 |W\|_{\dot{H}^1}^2\right| < \frac{1}{2} \|W\|_{\dot{H}^1(\{x: |x|>c_2\})}^2. 
 \]
\end{itemize}
A comparison of radiation profiles also shows that 
\[
 \|\vec{v}_{t,L}(0)-\vec{u}_L(t)\|_{\dot{H}^1\times L^2} \lesssim_1 \|G_-\|_{L^2([t,+\infty))} + \|G_+\|_{L^2((-\infty,-t])} \rightarrow 0, \quad t\rightarrow +\infty,
\]
which, as well as the soliton resolution as $t\rightarrow +\infty$, implies that 
\[
 \lim_{t\rightarrow +\infty} \|\vec{u}(t) - \vec{v}_{t,L}(0)\|_{\dot{H}^1\times L^2} = \lim_{t\rightarrow +\infty} \|\vec{u}(t) - \vec{u}_{L}(t)\|_{\dot{H}^1\times L^2} = \sqrt{J} \|W\|_{\dot{H}^1}. 
\]
As a result, $u(\cdot, \cdot+t)$ must be a $J$-bubble exterior solution for sufficiently large time $t$. It immediately follows from Proposition \ref{prop concentration} that 
\begin{align*}
 & \left(\sup_{0<r<\lambda_{J-1}(t)} \frac{\lambda_{J-1}(t)}{r} \int_{0}^r \left(|G_{t,+}(s)|^2 + |G_{t,-}(s)|^2 \right) {\rm d} s\right)^{1/2} + \|\chi_0 v_{t,L}\|_{Y(\Rm)} \\
  & \qquad + \sup_{r>0} \left(\frac{1}{r^{1/2}} \int_{0}^r \left(|G_{t,+}(s)| +|G_{t,-}(s)| \right) {\rm d} s \right) \geq \tau_3, \qquad \forall t \gg 1.
\end{align*}
Here $\lambda_{J-1}(t)$ is the size of the $(J-1)$-th bubble, as given in Proposition \ref{main tool}. Combining this with \eqref{relationship G Gt} and \eqref{global part tau bound}, we obtain 
\[
 \left(\sup_{0<r<\lambda_{J-1}(t)} \frac{\lambda_{J-1}(t)}{r} \int_{0}^r \left(|G_+(s-t)|^2 + |G_-(s+t)|^2 \right) {\rm d} s\right)^{1/2} > \frac{\tau_3}{2}, \qquad t\gg 1. 
\]
Now we let $g_+$ and $g_-$ be the (right hand side) maximal functions of $|G_+|^2$, $|G_-|^2$ respectively 
\begin{align} \label{def of g plus minus}
 & g_+(-t) = \sup_{r>0} \frac{1}{r} \int_{-t}^{-t+r} |G_+(s)|^2 {\rm d} s; & & g_-(t) = \sup_{r>0} \frac{1}{r} \int_{t}^{t+r} |G_-(s)|^2 {\rm d} s. 
\end{align}
The lower bound above implies that the inequality
\begin{equation} \label{lower bound of g plus minus}
 g_+(-t) + g_-(t) \geq \frac{\tau_3^2}{4} \lambda_{J-1}(t)^{-1}
\end{equation} 
holds for all time $t> T$. Here $T\geq R$ is a large time. We recall $|G_+(s)|^2 \in L^1(\Rm)$, $|G_-(s)|^2 \in L^1 ([R,+\infty))$ and the fact that the maximal function is of weak $(1,1)$ type to deduce that there exists a constant $C$, such that 
\begin{align*}
 \left|\left\{t \in [T,+\infty): g_+(-t) > \kappa \right\}\right| & \leq \frac{C}{\kappa}, & & \forall \kappa > 0; \\
 \left|\left\{t \in [T,+\infty): g_-(t) > \kappa \right\}\right| & \leq \frac{C}{\kappa}, & & \forall \kappa > 0.
\end{align*}
Here the notation $|\cdot|$ represents the Lebesgue measure of a subset of $\Rm$. A combination of these inequalities with \eqref{lower bound of g plus minus} yields
\begin{equation*}  
 \left|\left\{t\in [T,+\infty): \frac{\tau_3^2}{4} \lambda_{J-1}(t)^{-1} > 2\kappa\right\} \right| \leq \frac{2C}{\kappa}, \qquad \forall \kappa > 0. 
\end{equation*}
We may simplify it and write it in the form of 
\begin{equation}\label{estimate of lambda J1} 
 \left|\left\{t\in [T,+\infty): \lambda_{J-1}(t) < \eta \right\} \right| \leq C_\ast \eta, \qquad \forall \eta > 0. 
\end{equation} 
Here $C_\ast$ is a constant independent of $\eta > 0$. Since the soliton resolution implies that (see Remark \ref{match of lambda} below)
\[
 \lim_{t\rightarrow +\infty} \frac{\lambda_{J-1}(t)}{t} = 0,
\]
there exists a number $T_\ast > T$, such that 
\[
 \lambda_{J-1}(t) < \frac{t}{2 C_\ast}, \qquad t \geq T_\ast. 
\]
As a result, we always have $\lambda_{J-1}(t) < \eta$ for all $t \in [T_\ast, 2C_\ast \eta]$, as long as $\eta$ is sufficiently large. Thus we have 
\[
 \left|\left\{t\in [T,+\infty): \lambda_{J-1}(t) < \eta \right\} \right| \geq  2 C_\ast \eta - T_\ast, \qquad \forall \eta \gg 1. 
\]
This gives a contradiction with \eqref{estimate of lambda J1} and finishes the proof in the case of global solution. 

\paragraph{The type II blow-up case} We may assume the radiation part $u_L$ in the soliton resolution comes with a very small energy norm
\[
 \left\|\vec{u}_L(T_+)\right\|_{\dot{H}^1\times L^2} < \varepsilon. 
\]
by applying a cut-off technique if necessary.  Here $\varepsilon = \varepsilon(J) \ll \tau_2(J)$ is a sufficiently small constant. We might further reduce the upper bound of $\varepsilon$ in the argument below but the upper bound always depends on $J$ only. We may define the solution $u$ in the exterior region 
\[
 \{(x,t): |x|>|T-T_+|\}
\] 
by solving (CP1) with initial data $\vec{u}_L(T_+)$. Finite speed of propagation shows that the new exterior solution coincides with the original solution wherever both solutions are defined. By small data theory, we also have 
\[
 \sup_{t \in \Rm} \|\vec{u}(t)\|_{\mathcal{H}(|t-T_+|)} \leq  2\varepsilon. 
\]
Thus we may fix a time $t_0$ slightly smaller than $T_+$ and find a small number $r_0 > 0$, such that 
\[
 \|\vec{u}(t_0)\|_{\mathcal{H}(T_+ - t_0 - r_0)} \leq 3 \varepsilon. 
\]
Again the small data theory implies that $u$ can also be defined in the region $\{(x,t): t<t_0, |x|> T_+ - t - r_0\}$ with 
\[
 \sup_{t\leq t_0} \|\vec{u}(t)\|_{\mathcal{H}(T_+-t-r_0)} \leq 4 \varepsilon. 
\]
In summary, we may define the (nonlinear) radiation profile $G_\pm$ of $u$ with 
\begin{align} \label{upper bound of G plus minus type II}
  &\|G_+\|_{L^2(-T_+,+\infty)} \lesssim_1 \varepsilon, & & \|G_-\|_{L^2(T_+ - r_0, +\infty)} \lesssim_1 \varepsilon. 
\end{align}
As a result, the time-translated solution $u(\cdot, \cdot+t)$ is asymptotically equivalent to a free wave $v_{t,L}$, whose radiation profiles can be given by 
\begin{align*}
 &G_{t,+}(s) = G_+(s-t), \quad s > 0;& & G_{t,-}(s) = G_-(s+t), \quad s>0. 
\end{align*}
By the Strichartz estimates and \eqref{upper bound of G plus minus type II}, the following inequality
\begin{align} 
 \|\chi_0 v_{t,L}\|_{Y(\Rm)} + \sup_{t>0} & \left(\frac{1}{r^{1/2}} \int_{0}^r \left(|G_{t,+}(s)| +|G_{t,-}(s)| \right) {\rm d} s\right) \nonumber\\
  & \lesssim_1 \|G_+\|_{L^2(-t,+\infty)} + \|G_-\|_{L^2(t,+\infty)}
 \lesssim_1 \varepsilon. \label{type II part 1}
\end{align}
holds for $t \in [T_+ - r_0, T_+)$. Following a similar argument to the global case and using the continuity of $\|\vec{u}(t) - \vec{v}_{t,L}(0)\|_{\dot{H}^1\times L^2}$, we may show that $u(\cdot,\cdot +t)$ must be a $J$-bubble exterior solution for these times $t$, as long as $\varepsilon < \varepsilon(J)$ is sufficiently small. It follows from Proposition \ref{prop concentration}, our assumption $\varepsilon \ll \tau_3$ and \eqref{type II part 1} that 
\[
 \left(\sup_{0<r<\lambda_{J-1}(t)} \frac{\lambda_{J-1}(t)}{r} \int_{0}^r \left(|G_+(s-t)|^2 + |G_-(s+t)|^2 \right) {\rm d} s\right)^{1/2} > \frac{\tau_3}{2}, \qquad t\in [T_+ - r_0, T_+). 
\]
The same argument as in the case of global solutions shows that there exists a constant $C_\ast > 0$ such that 
\begin{equation}\label{estimate of lambda J1 2} 
 \left|\left\{t\in [T_+ - r_0, T_+): \lambda_{J-1}(t) < \eta \right\} \right| \leq C_\ast \eta, \qquad \forall \eta > 0. 
\end{equation} 
We recall that the following holds in the soliton resolution 
\[
 \lim_{t\rightarrow T_+} \frac{\lambda_{J-1}(t)}{T_+ - t} = 0,
\]
thus there exists a small constant $r_1 < r_0$, such that 
\[
 \lambda_{J-1}(t) < \frac{1}{2C_\ast} (T_+ - t), \qquad \forall t \in [T_+-r_1,T_+). 
\]
As a result, we always have $\lambda_{J-1}(t) < \eta$ for all $t \in [T_+ - 2C_\ast \eta, T_+)$, as long as $\eta$ is sufficiently small, which immediately gives a contradiction with \eqref{estimate of lambda J1 2} and finishes the proof in the type II blow-up case. 
\begin{remark} \label{match of lambda}
 The scale $\lambda_{J-1}(t)$ in the argument above is given by Proposition \ref{main tool}, which is not necessarily the same as the scale $\lambda_j^\ast(t)$ in the soliton resolution
 \[
  \vec{u}(t) = \sum_{j=1}^J \zeta_j^\ast (W_{\lambda_j^\ast(t)},0) + \vec{u}_L(t) + o(1), \qquad t\rightarrow T_+. 
 \]
 However, when $t$ is sufficiently close to $T_+$, we may apply Proposition \ref{main tool} and utilize the soliton resolution to deduce 
 \begin{align*}
  \left\|\sum_{j=1}^J \zeta_j^\ast (W_{\lambda_j^\ast(t)},0) - \sum_{j=1}^J \zeta_j (W_{\lambda_j(t)},0)\right\|_{\dot{H}^1\times L^2} & = \left\|\vec{u}(t) - \vec{u}_L(t)-o(1) - \sum_{j=1}^J \zeta_j (W_{\lambda_j(t)},0)\right\|_{\dot{H}^1\times L^2} \\
  & \leq \left\|\vec{u}(t) - \vec{v}_{t,L} (0) - \sum_{j=1}^J \zeta_j (W_{\lambda_j(t)},0)\right\|_{\dot{H}^1\times L^2} \\
  & \quad + \left\|\vec{v}_{t,L}(0) - \vec{u}_L(t) - o(1)\right\|_{\dot{H}^1\times L^2} \\
  & \lesssim_J \varepsilon;
 \end{align*}
 and ($j=1,2,\cdots,J-1$)
 \begin{align*}
  &\frac{\lambda_{j+1}(t)}{\lambda_j(t)} \lesssim_J \varepsilon^2;& & \frac{\lambda_{j+1}^\ast(t)}{\lambda_j^\ast(t)} < \varepsilon^2. 
 \end{align*}
 This immediately gives $\lambda_j (t)\simeq_1 \lambda_j^\ast (t)$, as long as $\varepsilon < \varepsilon(J)$ is sufficiently small. As a result, we still have 
 \[
  \lim_{t\rightarrow T_+} \frac{\lambda_{j}(t)}{T_+-t}  = 0, \qquad j=1,2,\cdots,J. 
 \]
 The situation of global solutions is similar (and even better). In fact, in the global case we have 
 \begin{align*}
  \lim_{t\rightarrow +\infty} \left\|\sum_{j=1}^J \zeta_j^\ast (W_{\lambda_j^\ast(t)},0) - \sum_{j=1}^J \zeta_j (W_{\lambda_j(t)},0)\right\|_{\dot{H}^1\times L^2} & = 0; \\
  \lim_{t\rightarrow +\infty} \left(\frac{\lambda_{j+1}(t)}{\lambda_j(t)} + \frac{\lambda_{j+1}^\ast(t)}{\lambda_j^\ast(t)}\right) & = 0, \quad j=1,2,\cdots,J-1; 
 \end{align*}
 which implies that 
 \[
  \lim_{t\rightarrow +\infty} \frac{\lambda_j(t)}{\lambda_j^\ast(t)} = 1; \qquad \Longrightarrow \qquad \lim_{t\rightarrow +\infty} \frac{\lambda_j(t)}{t} = 0,\quad j=1,2,\cdots,J.  
 \]
\end{remark}
\section*{Acknowledgement}
The author is financially supported by National Natural Science Foundation of China Project 12471230.

\end{document}